\documentclass[11pt,reqno]{amsart}

\usepackage{amsfonts,amssymb,amsmath,epic,eepic,float,fullpage}
\usepackage{longtable}
\usepackage[usenames,dvipsnames]{color}
\usepackage{siunitx}
\usepackage{empheq}
\usepackage[mathscr]{euscript}
\usepackage[T1]{fontenc}
\usepackage{mathtools}
\usepackage{comment}
\usepackage{bm}
\usepackage{caption}
\usepackage{hyperref}
\usepackage{enumerate,enumitem}
\usepackage{tikz}
\usetikzlibrary{graphs,patterns,decorations.markings,arrows,matrix}
\usetikzlibrary{calc,decorations.pathmorphing,decorations.pathreplacing,shapes}
\allowdisplaybreaks
\usetikzlibrary{arrows.meta,positioning}

\usetikzlibrary{arrows.meta} 

\newtheorem{thm}{Theorem}[section]

\newtheorem{prop}[thm]{Proposition}

\newtheorem{cor}[thm]{Corollary}

\theoremstyle{definition}

\theoremstyle{remark}
\newtheorem{rmk}[thm]{Remark}

\theoremstyle{remark}
\newtheorem{ex}[thm]{Example}

\def\leq{\leqslant}
\def\geq{\geqslant}

\colorlet{lgray}{white!85!black}
\colorlet{lred}{white!85!red}
\colorlet{lgreen}{white!80!green}
\colorlet{dgreen}{black!30!green}
\definecolor{green}{rgb}{0.1,0.8,0.1}
\definecolor{yellow}{rgb}{1.0,0.85,0.25}

\renewcommand{\tikz}[2]{
	\begin{tikzpicture}[scale=#1,baseline=(current bounding box.center),>=stealth]
		#2
\end{tikzpicture}}

\def\C{\mathcal{C}}

\def\I{\bm{I}}

\newcommand{\Is}[2]{\I_{[#1,#2]}}

\renewcommand\le{\leq}
\renewcommand\ge{\geq}

\def\O{\mathcal{O}}

\def\Z{\mathbb Z}
\def\C{\mathbb C}
\def\R{\mathbb R}
\def\A{\mathcal A}
\def\M{\mathcal M}

\newcommand{\eps}{\varepsilon}

\newcommand{\ii}{{\mathbf i}}

\numberwithin{equation}{section}

\makeindex

\begin{document}
	
	\title{Leaders in multi-type TASEP}
	
	\author{Alexei Borodin}\address[Alexei Borodin]{Department of Mathematics, Massachusetts Institute of Technology, U.S.A.}\email{borodin@math.mit.edu}
	
	\author{Alexey Bufetov}\address[Alexey Bufetov]{Institute of Mathematics, Leipzig University, Germany.}\email{alexey.bufetov@gmail.com}

	\begin{abstract}
		We study the totally asymmetric simple exclusion process (TASEP) on 
		$\mathbb{Z}$ with step initial condition, in which all particles have distinct types. Our main object of interest is the type of the rightmost particle --- the leader --- at large time $t$. We prove a central limit theorem for this random variable. Somewhat unexpectedly, the problem is closely connected to certain observables of voter and coalescing processes on $\mathbb{Z}$; we therefore derive their asymptotics as well. We also analyze the large-time behavior of a few other related observables, including certain multi-particle ones.
	\end{abstract}
	
	\maketitle
	
	\setcounter{tocdepth}{1}
	\makeatletter
	\def\l@subsection{\@tocline{2}{0pt}{2.5pc}{5pc}{}}
	\makeatother
	\tableofcontents
	
	\section{Introduction}
	
	\subsection{Overview}
	
	The main object of study in this paper is the multi-type Totally Asymmetric Simple Exclusion Process (TASEP). This is an interacting particle system on $\mathbb{Z}$ whose states can be viewed as maps $\eta:\mathbb{Z}\to \mathbb{Z}\cup\{-\infty\}$. The value $\eta(x)=c\in \mathbb{Z}$ is interpreted as a particle of type $c$ occupying position $x$, while $\eta(y)=-\infty$ indicates the absence of a particle at location $y$, i.e., a hole.
	
	The continuous-time Markov evolution $(\eta_t)_{t\ge 0}$ is defined as follows. Each ordered pair of neighboring sites in $\mathbb{Z}$ is equipped with an independent exponential clock of rate $1$. When the clock associated with $(x,x+1)$ rings, we swap the values at $x$ and $x+1$ if they are in strictly decreasing order, and do nothing otherwise.
	
	We are primarily interested in the initial configuration
	\[
	\eta_0(x)=-x,\qquad x\le 0,
	\qquad\text{and}\qquad
	\eta_0(x)=-\infty,\qquad x>0,
	\]
	which is a multi-type version of the step initial condition, see Figure~\ref{fig:intro-1} for an example. Since particles of larger type have higher priority, the type of the right-most particle (the \textit{leader}) is nondecreasing under the dynamics. Let $L_1(t)$ denote the type of the leader at time $t$. One of our main results is the following theorem.
	
	\begin{thm}[Theorem \ref{th:leader-type} below]
		We have
		$$
		\frac{L_1(t)}{\sqrt{t}} \xrightarrow[t \to \infty]{} \frac{1}{\sqrt{\pi}} \exp \left( \frac{-a^2}{4} \right) \mathbf{1}_{a \ge 0} \, da, \qquad \mbox{in distribution}.
		$$
	\end{thm}
	
	Let $\mathcal{S}(t)$ be the number of changes of the leader's type up to time $t$. We establish the asymptotic behavior of its expectation.
	
	\begin{thm}[Theorem \ref{th:leader-changes} below]
		We have
		$$
		\lim_{t \to \infty} \frac{ \mathbf{E} \left( \mathcal{S}(t) \right)}{\ln(t)} = \frac{3 \sqrt{3}}{4 \pi}.
		$$
	\end{thm}
	
	In addition, we prove several other asymptotic results for related quantities:
	
	\begin{itemize}
		\item In Theorem~\ref{th:leader-conditioned-asymp} we establish the joint limiting distribution of the position of the leader and its type.
		\item In Theorem~\ref{th:joint} we find the joint limiting distribution of the types of the \textit{two} leading (right-most) particles.
		\item In Section~\ref{sec:Voter-section} we prove an equality in distribution (a duality) between the leader's type and certain observables of the voter process on $\mathbb{Z}$. This allows us to translate the asymptotic results above to the voter-process setting; see, for instance, Corollaries~\ref{cor:asymp-voter} and \ref{cor:coalescence}.
		\item In Section~\ref{sec:Ranking} we introduce a new natural process --- the \textit{TASEP ranking process} --- whose behavior is closely related to the asymptotic behavior of leaders; see Proposition~\ref{prop:ranking-duality}.
	\end{itemize}
	
	Our approach is based on explicit formulas for $q$-moments of the stochastic six-vertex model (see Sections~\ref{sec:integral} and \ref{sec:observables}) and their $q\to 0$ degenerations.

	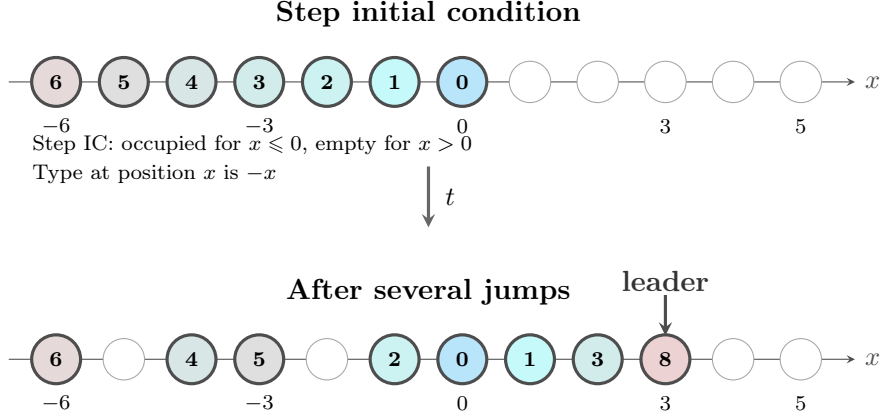
\begin{figure}
		\begin{tikzpicture}[x=0.9cm,y=0.9cm,>=stealth,font=\small]
			
			\tikzset{
				site/.style={circle, draw=black!35, fill=white, minimum size=5.5mm, inner sep=0pt},
				part/.style={circle, draw=black!70, very thick, minimum size=6.5mm, inner sep=0pt},
			}
			
			\newcommand{\particle}[3]{%
				\pgfmathsetmacro{\mix}{min(70,10*#3)}%
				\node[part, fill=red!\mix!cyan!25!white] at (#1,#2) {\scriptsize\textbf{#3}};%
			}
			
			\def\xmin{-6}
			\def\xmax{5}
			
			\node[font=\bfseries] at (-0.5,3.2) {Step initial condition};
			
			\draw[->, black!70] (\xmin-0.7,2.2) -- (\xmax+0.8,2.2) node[right] {$x$};
			
			\foreach \x in {\xmin,...,\xmax} {
				\node[site] at (\x,2.2) {};
			}
			\node at (\xmin,1.55) {\scriptsize $\xmin$};
			\node at (-3,1.55)    {\scriptsize $-3$};
			\node at (0,1.55)     {\scriptsize $0$};
			\node at (3,1.55)     {\scriptsize $3$};
			\node at (\xmax,1.55) {\scriptsize $\xmax$};
			
			\foreach \x in {\xmin,...,0} {
				\pgfmathtruncatemacro{\ty}{-1*(\x)} 
				\particle{\x}{2.2}{\ty}
			}
			
			\node[align=left, anchor=west] at (\xmin-0.5,1.05)
			{\scriptsize Step IC: occupied for $x\le 0$, empty for $x>0$\\[-1pt]\scriptsize Type at position $x$ is $-x$};
			
			\draw[->, very thick, black!60] (-0.5,0.95) -- (-0.5,0.05);
			\node[right] at (-0.4,0.5) {$t$};
			
			\node[font=\bfseries] at (-0.5,-0.9) {After several jumps};
			
			\draw[->, black!70] (\xmin-0.7,-1.9) -- (\xmax+0.8,-1.9) node[right] {$x$};
			
			\foreach \x in {\xmin,...,\xmax} {
				\node[site] at (\x,-1.9) {};
			}
			\node at (\xmin,-2.55) {\scriptsize $\xmin$};
			\node at (-3,-2.55)    {\scriptsize $-3$};
			\node at (0,-2.55)     {\scriptsize $0$};
			\node at (3,-2.55)     {\scriptsize $3$};
			\node at (\xmax,-2.55) {\scriptsize $\xmax$};
			
			\particle{-6}{-1.9}{6}
			\particle{-4}{-1.9}{4}
			\particle{-3}{-1.9}{5}
			\particle{-1}{-1.9}{2}
			\particle{0}{-1.9}{0}
			\particle{1}{-1.9}{1}
			\particle{2}{-1.9}{3}
			\particle{3}{-1.9}{8} 
			
			\draw[->, very thick, black!70] (3,-0.95) -- (3,-1.55);
			\node[font=\bfseries, black!80] at (3,-0.75) {leader};
			
			
			
		\end{tikzpicture}
		\caption{TASEP with the multi-type step initial conditon (top) and a possible configuration of the TASEP after several jumps (bottom).}
		\label{fig:intro-1}
	\end{figure}

	\subsection{Background} 
	
	The totally asymmetric simple exclusion process (TASEP) is a prototypical interacting particle system on $\mathbb{Z}$, originating in the classical work of Spitzer \cite{Spi}. Since the early 1970s, TASEP and its variants have been studied extensively from probabilistic, hydrodynamic, and integrable perspectives; we refer to the monographs \cite{Lig,Lig2,KL,Spo} and the review \cite{Der} for background and further references. A particularly important integrable regime is the TASEP with step initial condition. In this setting, Rost proved a first nontrivial law of large numbers and local equilibrium results \cite{Ros}. Later, Johansson identified the $t^{1/3}$-scale fluctuations and the GUE Tracy--Widom limit using the connection to last passage percolation \cite{Joh}; see also \cite{PS,TW,Sch} and many other works that followed.
	
	Multi-type (or multi-species, or colored; we use these terms interchangeably) TASEP is a rich extension of the single-species model that has also attracted substantial attention. Beyond its intrinsic interest, there are at least two broad sources of motivation for studying this refined setting. First, \emph{second class} particles (and, more generally, discrepancies under the basic coupling) serve as microscopic probes of characteristics and shock/rarefaction phenomena; see, for instance, \cite{FK,MG,FP,BF}, and references therein. Second, multi-type systems exhibit a richer and often more natural algebraic and integrable structure, see, e.g., \cite{FM,AAV,Buf,BW-coloured,BW-observables,BB}.
	
	To the best of our knowledge, the present work is the first to study the \emph{types of the leading particles} in the multi-type TASEP started from the step initial condition. We address what we view as the most basic questions in this model. Obviously, our results are in no way exhaustive; see, for example, Remark~\ref{rem:permut} for a natural direction for further study. We hope that this paper will draw attention to this natural generalization of a classical and extensively studied integrable model.

	\subsection*{Acknowledgments}
	We are grateful to Oleg Zaboronski for helpful discussions. 
	A.~Borodin was partially supported by the NSF grant
	DMS-2450323, and the Simons Investigator program.
	A.~Bufetov was partially supported by the European Research Council (ERC), Grant Agreement No. 101041499.
	
	\section{Integral representation}\label{sec:integral}
	
	Our main object of interest is the so-called colored Asymmetric Simple Exclusion Process (or ASEP, for short), which is also known under the names of multi-type and multi-species ASEP. This is an interacting particle system on $\Z$ 
	whose states can be viewed as maps $\eta:\Z\to\Z\cup\{-\infty\}$. The value $\eta(x)=c\in \Z$ is interpreted as a particle of color $c$ occupying position $x$. The value $\eta(y)=-\infty$ is interpreted as the absence of a particle at location $y$. Denote by $\mathfrak C$ the set of all such states or particle configurations $\eta$.  
	
	The Markovian evolution of the colored ASEP is defined as follows. Each particle has two exponential clocks -- left and right -- of rates $L$ and $R$, respectively; all clocks are independent. When a clock rings, the particle attempts to move in the corresponding direction by a single unit. If the target site is occupied by a particle of a weakly higher color, the jump is suppressed. On the other hand, if it is occupied by a particle of a strictly lower color or not occupied at all, our particle swaps its position with the particle/hole at its target location. The fact that this is a well-defined Markov process follows from the well-known graphical construction of \cite{Harris1965}, \cite{Liggett1976}.
	
	We will be mostly interested in the case of \emph{Totally} ASEP, or TASEP, with $L=0$. Thus, in the TASEP only right jumps are initiated by the clocks. Without loss of generality, one can then set $R=1$. Whenever a colored ASEP shows up below, we will still assume that its right jump rate $R$ is equal to $1$, while we will re-denote its left jump rate $L$ by $q$. 
	
	For $\eta\in\mathfrak C$ and a color $c\in\Z$, denote by $\A_c(\eta)$ the subset of $\Z$ such that $x\in \A_c(\eta)$ if and only if $\eta(x)\ge c$. In other words, $\A_c(\eta)$ is the set of positions of particles of colors that are no less than $c$. 
	
	Take $n\ge1$, and let $C=(c_1<\dots<c_n)\in \Z^n$ be an ordered $n$-tuple of integers (that are to be used momentarily as color cutoffs). Let $\eta\in\mathfrak C$ be such that $|\A_{c_i}(\eta)|\ge n-i+1$ and $\max(\A_{c_i}(\eta))<\infty$ for each $i=1,\dots,n$. In other words, for each color $c_i$, there are at least $(n-i+1)$ particles of color no less than $c_i$, and among those particles there is a right-most one (the latter condition is nontrivial as $\A_{c_i}(\eta)$ may be infinite). Note that if $\eta$ satisfies these conditions, any configuration that can be reached via time evolution started from $\eta$ with nontrivial probability will also satisfy them. 
	
	Given such $C\in \Z^n$ and $\eta\in\mathfrak C$, define $\M_C(\eta)=(k_1,\dots,k_n)\in\Z^n$ as follows. For $n,n-1,\dots,1$ in that order we recursively set
	\begin{equation}\label{eq:maximums}
		k_n=\max(\A_{c_n}(\eta)), \qquad k_i=\max\bigl(\A_{c_i}(\eta)\setminus \{k_{i+1},\dots,k_n\}\bigr), \quad i=n-1,\dots,1.
	\end{equation}
	In other words, $k_i$ records the maximum of $\A_{c_i}(\eta)$ disregarding the maxima that have already been recorded for larger color cutoffs $c_{i+1},\dots,c_n$. The assumptions of the previous paragraph guarantee that the definition of $\M_C(\eta)$ makes sense. Further, the definition implies that the coordinates of $\M_C(\eta)$ are necessarily pairwise distinct. In what follows we will also use the notation $\M_C(t)$ instead of $\M_C(\eta)$ if $\eta$ is undergoing a Markovian evolution with time parameter $t$.  
	
	In general, some of the maxima of $\A_{c_i}(\eta)$ may coincide. The definition of $\M_C(\eta)$ could be viewed as a natural way of modifying the sequence of those maxima in such a way that all its coordinates remain distinct. 
	
	We are interested in controlling the distribution of $\M_C(t)$. In order to state the main result of this section that will give a formula for this distribution, we need to introduce additional notation. 
	
	A family of rational functions 
	$$
	\varphi_\mu: \C^n \to \C, \qquad \mu\in \Z^n,\  \mu_i\ne \mu_j,
	$$
	indexed by $n$-tuples $\mu$ of unequal integers, is defined as follows. If $\delta=(\delta_1<\dots<\delta_n)\in\Z^n$ then 
	\begin{equation}\label{eq:antidominant}
		\varphi_\delta (w_1,\dots,w_n)=\prod_{i=1}^n \frac1{(1+w_i)^{\delta_i}}.
	\end{equation}
	Further, if $\mu$ is such that $\mu_i<\mu_{i+1}$ for some $1\le i\le n-1$, then 
	\begin{equation}
		\label{eq:phi-exchange}
		\varphi_{(\mu_1,\dots,\mu_{i+1},\mu_i,\dots,\mu_n)}(w_1,\dots,w_n)=\frac{w_{i+1}(1+w_i)}{w_{i}-w_{i+1}}\,(1-\mathfrak{s}_i)
		\cdot \varphi_\mu(w_1,\dots,w_n)
	\end{equation}
	with elementary transpositions $\mathfrak s_i\cdot h(w_1,\dots,w_n):=h(w_1,\dots,w_{i+1},w_i,\dots,w_n)$. This recurrence relation allows to extend the definition \eqref{eq:antidominant} to all $\mu$'s. The fact that such an extension is well-defined is nontrivial, and it will follow from the proof of Proposition \ref{prop:TASEP-transition} below. 
	
	We also need another family of rational functions
	$$
	\psi_\nu: \C^n \to \C, \qquad \nu\in \Z^n,\  \nu_i\ne \nu_j,
	$$
	indexed by the same set and defined in a similar fashion. For $\epsilon=(\epsilon_1>\dots>\epsilon_n)\in\Z^n$ we set
	\begin{equation}\label{eq:dominant}
		\psi_{\epsilon}(w_1,\dots,w_n)=\prod_{i=1}^n (1+w_i)^{\epsilon_i}.
	\end{equation}
	This is extended to arbitrary $\nu$'s using the following recurrence relation: If $\nu_i>\nu_{i+1}$ for some $1\le i\le n-1$, then 
	\begin{equation}
		\label{eq:psi-exchange}
		\psi_{(\nu_1,\dots,\nu_{i+1},\nu_i,\dots,\nu_n)}(w_1,\dots,w_n)=\frac{w_i(1+w_{i+1})-w_{i+1}(1+w_i)\mathfrak s_i}{w_{i}-w_{i+1}}\cdot\psi_\nu(w_1,\dots,w_n).
	\end{equation}
	
	As for $\varphi_\mu$'s, the correctness of this definition is not obvious, and it will follow from the proof of Proposition \ref{prop:TASEP-transition}.
	
	We are now ready to state an integral representation for the distribution of $\M_C(t)$.  
	
	\begin{thm} \label{thm:integral}
		Let $\eta_0\in\mathfrak C$ be a (random or deterministic) initial condition for the colored TASEP, and let $C=(c_1<\dots<c_n)$, $n\ge 1$, be a sequence of color cutoffs. Assume that $\M_C(0)=\M_C(\eta_0)$ is well-defined and deterministic, cf. the assumptions above \eqref{eq:maximums}; denote $\M_C(0)$ by $\nu$. Then 
		\begin{equation}\label{eq:integral}
			\begin{gathered}
				\mathrm{Prob}\{\M_C(t)=\mu\}=\frac{1}{(2\pi i)^n}\oint\cdots\oint_{|w_i|=R>1} \prod_{1\le i<j\le n} \left(1-\frac{w_i}{w_j}\right) \\ \times \varphi_\mu(w_1,\dots,w_n)\psi_\nu(w_1,\dots,w_n)\prod_{j=1}^n \frac{e^{tw_j} dw_j}{(1+w_j)^{n-j+1}}\,,
			\end{gathered}
		\end{equation}
		where the integral is over positively oriented circles $|w_i|=R$, $i=1,\dots,n$ with $R>1$. 
	\end{thm}
	
	As we will only perform asymptotic analysis when $\psi_\nu$ is factorized as in \eqref{eq:dominant}, it is convenient to have a simplified formula in that case explicitly written. 
	
	\begin{cor}\label{cor:integral}
		In the notation of Theorem \ref{thm:integral}, assume that the coordinates of $\M_C(0)=\nu$ are decreasing: $\nu=(\nu_1>\dots>\nu_n)$. Then 
		\begin{equation}\label{eq:integral-dominant}
			\begin{gathered}
				\mathrm{Prob}\{\M_C(t)=\mu\}=\frac{1}{(2\pi i)^n}\oint\cdots\oint_{|w_i|=R>1} \prod_{1\le i<j\le n} \left(1-\frac{w_i}{w_j}\right) \\ \times \varphi_\mu(w_1,\dots,w_n)\prod_{j=1}^n 
				{(1+w_j)^{\nu_j-n+j-1}}{e^{tw_j} dw_j}\,,
			\end{gathered}
		\end{equation}
		where the integral is over positively oriented circles $|w_i|=R$, $i=1,\dots,n$ with $R>1$. 
	\end{cor}
	
	In view of the relative simplicity of the expressions in the right-hand sides of \eqref{eq:integral}-\eqref{eq:integral-dominant}, it is natural to ask where the definition of the left-hand sides came from. Indeed, it is not obvious that $\M_C(t)$ is a good object to consider; for us it appeared through a degeneration procedure of the observable for the colored stochastic vertex models introduced in \cite{BW-observables}. A brief explanation of that degeneration can be found in Section \ref{sec:observables} below. 
	
	The proof of Theorem \ref{thm:integral} consists of two steps that we state as two propositions below. 
	
	\begin{prop}\label{prop:duality}
		In the notation of Theorem \ref{thm:integral}, $\mathrm{Prob}\{\M_C(t)=\mu\}$ is equal to the probability that in a colored TASEP with $n$ particles of colors $1,\dots,n$ that are initially (at time $t=0$) at positions $\nu_1,\dots,\nu_n$, respectively, the positions of those particles at time $t$ are $\mu_1,\dots,\mu_n$, respectively. 
		In other words, this is exactly the transition probability of the $n$-particle TASEP with particles of different colors.   
	\end{prop}
	
	In particular, this means that $\M_C(\eta)$ can be viewed as a \emph{duality functional} relating the colored TASEP of Theorem \ref{thm:integral}, which may have infinitely many particles, and the $n$-particle colored TASEP of Proposition \ref{prop:duality}.
	
	\begin{prop}\label{prop:TASEP-transition}
		The transition probability over time $t$ of the colored TASEP with $n$ particles of colors $1,\dots,n$ between initial and final positions $\nu=(\nu_1,\dots,\nu_n)$ and $\mu=(\mu_1,\dots,\mu_n)$, respectively, is given by the right-hand side of \eqref{eq:integral}. 
	\end{prop}	
	
	A colored ASEP version of the latter claim (for $\nu=(\nu_1>\dots>\nu_n)$, as in Corollary \ref{cor:integral}) has appeared as Theorem 3.5 in \cite{deGier-Mead-Wheeler}. Both formulas are suitable degenerations of an integral representation for transition probabilities (or entries of a transfer matrix) of the colored stochastic vertex models proved as \cite[(9.5.2)]{BW-coloured}. See the proof of Proposition \ref{prop:TASEP-transition} below for more details. 
	
	\begin{proof}[Proof of Proposition \ref{prop:duality}] It suffices to check that $\M_C(t)$ undergoes a Markovian evolution with jump rates that agree with those of the $n$-particle TASEP from the hypothesis. Consider a time moment $t\ge 0$, and denote by $\eta_t$ the state of the original TASEP of Theorem \ref{thm:integral} at $t$. Let $\M_C(t)=(k_1,\dots,k_n)$, choose $i$ such that $1\le i\le n$, and consider $k_i$. By the definition of $\M_C$, $\eta_t(k_i)\ge c_i$; that is, the position $k_i$ is occupied by a particle of color at least $c_i$. By the same definition, $k_i$ cannot change unless that particle changes its position, i.e., moves either left or right by 1. Let us consider all possible scenarios.    
		
		The particle at $k_i$ can only move left if $\eta_t(k_i-1)>\eta_t(k_i)$, in which case the rate of such move is equal to 1. If there is no $j>i$ such that $k_j=k_i-1$ then the corresponding swap of particles at $k_i-1$ and $k_i$ would not change $\A_{c_i}(\eta_t)\setminus\{k_{i+1},\dots,k_n\}$, which means that $k_i$ would not move. On the other hand, if $k_j=k_i-1$ for some $j>i$, then this particle swap would result in $\A_{c_i}(\eta_t)\setminus\{k_{i+1},\dots,k_n\}$ losing $k_i$ and gaining $k_{i}-1$. Since $k_i$ was previously the maximum of this set, the new maximum will be equal to $k_i-1$. Hence, we conclude that $k_i$ moves to the left by 1 with rate 1 if there is a $k_i-1=k_j$ with $j>i$ (note that the latter condition also implies that $\eta_t(k_i-1)=\eta_t(k_j)>\eta_t(k_i)$ as otherwise $\A_{c_j}(\eta_t)\setminus\{k_{j+1},\dots,k_n\}$ would contain a member that is larger than $k_j$). Further, as $k_i$ decreases by 1, $k_j$ increases by 1; i.e., the values of $k_i$ and $k_j$ are swapped. 
		
		Let us now look at the possibility of the particle at $k_i$ moving to the right. If there is no $j>i$ such that $k_j=k_i+1$,
		then we must have $\eta_t(k_i+1)<\eta_t(k_i)$, as otherwise $\A_{c_i}(\eta_t)\setminus\{k_{i+1},\dots,k_n\}$ would have a maximum greater than $k_i$. Hence, the right jump rate of the particle at $k_i$ is equal to 1, and if it jumps then $k_i= \max(\A_{c_i}(\eta_t)\setminus\{k_{i+1},\dots,k_n\})$ also increases by 1. Assume now that there is a $j>i$ such that $k_j=k_i+1$. Then either $\eta_t(k_i+1)\ge \eta_t(k_i)$, in which case the particle at $k_i$ cannot move to the right, or $\eta_t(k_i+1)< \eta_t(k_i)$. In the latter case, even though the particle at $k_i$ can move to the right, the value of $k_i$ is not going to change with that move. Indeed, such a move would change neither $\A_{c_j}(\eta)$, nor $\A_{c_j}(\eta)\setminus\{k_{j+1},\dots,k_m\}$, nor $k_j$, with the latter remaining equal to $k_i+1$. Further, $\eta_t(k_i+1)=\eta_t(k_j)\ge c_j>c_i$, which means that the particle landing at $k_i$ after the particle swap would still contribute to $\A_{c_i}(\eta)$, with $\A_{c_i}(\eta)\setminus\{k_{i+1},\dots,k_n\}$ and its maximum $k_i$ remaining unchanged. 
		Summarizing, $k_i$ moves to the right by 1 with rate 1 if and only if there is no $j>i$ such that $k_j=k_i+1$. If at the time of such a move there is an $l<i$ with $k_l=k_i+1$, then the values of $k_i$ and $k_l$ are swapped. 
		
		The above considerations show that the time evolution of $(k_1,\dots,k_n)$ is exactly that of the $n$-particle colored TASEP with increasing colors of the particles, as was desired.  
	\end{proof}
	
	\begin{proof}[Proof of Proposition \ref{prop:TASEP-transition}] It is possible to verify directly that the integral on the right-hand side of \eqref{eq:integral} is a unique solution of the Chapman-Kolmogorov equations for the $n$-particle colored TASEP with an appropriate initial condition. However, since this was already done in a substantially more general situation, it is more straightforward to take a limit of that more general result. 
		
		We are going to approximate the $n$-particle colored TASEP by a discrete time Markov chain. We start from a stochastic matrix whose rows and columns are indexed by elements of $\Z^n$ with unequal coordinates. They are explicitly given as partition functions of a doubly-infinite  single row vertex model that we symbolically picture as 
		\begin{equation}
			\label{eq:transfer-matrix}
			\tikz{0.6}{
				\draw[lgray,line width=1.5pt,<-] (1,0) -- (17,0);
				\foreach\x in {3,5,7,9,11,13,15}{
					\draw[lgray,line width=4pt,->] (\x,-1) -- (\x,1);
				}
				\node[left] at (1,0) {\tiny $0$};\node[right] at (17,0) {\tiny $0$};
				\node[below] at (3,-1) {\tiny $\cdots$};\node[above] at (3,1) {\tiny $\cdots$};
				\node[below] at (15,-1) {\tiny $\cdots$};\node[above] at (15,1) {\tiny $\cdots$};
				\node[below] at (9,-1) {\tiny $\I^0(\nu)$};\node[above] at (9,1) {\tiny $\I^0(\mu)$};
				\node[below] at (7,-1) {\tiny $\I^{-1}(\nu)$};\node[above] at (7,1) {\tiny $\I^{-1}(\mu)$};
				\node[below] at (5,-1) {\tiny $\I^{-2}(\nu)$};\node[above] at (5,1) {\tiny $\I^{-2}(\mu)$};
				\node[below] at (11,-1) {\tiny $\I^{1}(\nu)$};\node[above] at (11,1) {\tiny $\I^{1}(\mu)$};
				\node[below] at (13,-1) {\tiny $\I^{2}(\nu)$};\node[above] at (13,1) {\tiny $\I^{2}(\mu)$};
			}	
		\end{equation}
		
		Each vertex in our vertex model and in the figure above has four incident edges. The top and the bottom ones are marked by fixed vectors in $\Z^n$ with coordinates from $\{0,1\}$. If $\nu$ and $\mu$ are the initial and final states of a step  of the discrete time Markov chain we are constructing, then we set
		$$
		(\I^k(\nu))_j=\begin{cases} 1,& \text{if  } \nu_j=k,\\ 0,& \text{otherwise}, 
		\end{cases}
		\qquad
		(\I^k(\mu))_j=\begin{cases} 1,& \text{if  } \mu_j=k,\\ 0,& \text{otherwise}, 
		\end{cases}
		\qquad k\in\Z,\ j=1,\dots,n. 
		$$ 
		The left and right edges are marked by elements in $\{0,1,\dots,n\}$, and 0's on the left and on the right of the row symbolize that those marks have to be all equal to 0 sufficiently far to the left and sufficiently far to the right. The partition function of such a row is defined as the sum of the product of vertex weights over all the vertices, summed over all possible assignments of states of the horizontal edges. With the weights that we will use, it is easy to see that there is no more than one such assignment that leads to all vertices having a nonzero weight. Furthermore, in such an assignment the weights of all but finitely many vertices will be necessarily equal to 1, making the definition of the partition function meaningful. 
		
		The exact values of the vertex weights are given by the table below. If a vertex does not fit into one of the six pictured types, its weight is set to 0. 
		
		\begin{align}
			\label{eq:vertex-weights}
			\begin{tabular}{|c|c|c|}
				\hline
				\quad
				\tikz{0.7}{
					\draw[lgray,line width=1.5pt,<-] (-1,0) -- (1,0);
					\draw[lgray,line width=4pt,->] (0,-1) -- (0,1);
					\node[left] at (-1,0) {\tiny $0$};\node[right] at (1,0) {\tiny $0$};
					\node[below] at (0,-1) {\tiny $\I$};\node[above] at (0,1) {\tiny $\I$};
				}
				\quad
				&
				\quad
				\tikz{0.7}{
					\draw[lgray,line width=1.5pt,<-] (-1,0) -- (1,0);
					\draw[lgray,line width=4pt,->] (0,-1) -- (0,1);
					\node[left] at (-1,0) {\tiny $i$};\node[right] at (1,0) {\tiny $i$};
					\node[below] at (0,-1) {\tiny $\I$};\node[above] at (0,1) {\tiny $\I$};
				}
				\quad
				&
				\quad
				\tikz{0.7}{
					\draw[lgray,line width=1.5pt,<-] (-1,0) -- (1,0);
					\draw[lgray,line width=4pt,->] (0,-1) -- (0,1);
					\node[left] at (-1,0) {\tiny $i$};\node[right] at (1,0) {\tiny $0$};
					\node[below] at (0,-1) {\tiny $\I$};\node[above] at (0,1) {\tiny $\I^{-}_i$};
				}
				\quad
				\\[1.3cm]
				\quad
				$\dfrac{q^{-\Is{1}{n}}-q^{-1} y}{1-q^{-1} y}$
				\quad
				& 
				\quad
				$\dfrac{(1-y q^{-I_i}) q^{-\Is{i+1}{n}}}{1-q^{-1} y}$
				\quad
				& 
				\quad
				$\dfrac{(1-q^{-I_i}) q^{-\Is{i+1}{n}}}{1-q^{-1} y}$
				\quad
				\\[0.7cm]
				\hline
				\quad
				\tikz{0.7}{
					\draw[lgray,line width=1.5pt,<-] (-1,0) -- (1,0);
					\draw[lgray,line width=4pt,->] (0,-1) -- (0,1);
					\node[left] at (-1,0) {\tiny $0$};\node[right] at (1,0) {\tiny $i$};
					\node[below] at (0,-1) {\tiny $\I$};\node[above] at (0,1) {\tiny $\I^{+}_i$};
				}
				\quad
				&
				\quad
				\tikz{0.7}{
					\draw[lgray,line width=1.5pt,<-] (-1,0) -- (1,0);
					\draw[lgray,line width=4pt,->] (0,-1) -- (0,1);
					\node[left] at (-1,0) {\tiny $j$};\node[right] at (1,0) {\tiny $i$};
					\node[below] at (0,-1) {\tiny $\I$};\node[above] at (0,1) {\tiny $\I^{+-}_{ij}$};
				}
				\quad
				&
				\quad
				\tikz{0.7}{
					\draw[lgray,line width=1.5pt,<-] (-1,0) -- (1,0);
					\draw[lgray,line width=4pt,->] (0,-1) -- (0,1);
					\node[left] at (-1,0) {\tiny $i$};\node[right] at (1,0) {\tiny $j$};
					\node[below] at (0,-1) {\tiny $\I$};\node[above] at (0,1) 
					{\tiny $\I^{+-}_{ji}$};
				}
				\quad
				\\[1.3cm]
				\quad
				$\dfrac{y(q^{-\Is{1}{n}}-q^{-1})}{1-q^{-1} y}$
				\quad
				& 
				\quad
				$\dfrac{(1-q^{-I_j}) q^{-\Is{j+1}{n}}}{1-q^{-1} y}$
				\quad
				&
				\quad
				$\dfrac{y(1-q^{-I_i})q^{-\Is{i+1}{n}}}{1-q^{-1} y}$
				\quad
				\\[0.7cm]
				\hline
			\end{tabular} 
		\end{align}
		In this table we assume that $i,j \in \{1,\dots,n\}$ with $i<j$. The notation $\I_{[a,b]}$ stands for $\I_a+\dots+\I_b$, and $\I^{+-}_{kl}=\I+\mathbf{e}_k-\mathbf{e}_l$, where $\mathbf{e}_m$ is the standard basis vector with the single nonzero $m$th coordinate equal to 1, $1\le m\le n$.  
		Here $y$ and $q$ are real-valued parameters, with $y>1$, $0<q<1$.
		These weights coincide with those of \cite[(3.31)]{deGier-Mead-Wheeler}, which, in their turn, were obtained from the $\widetilde M_x$-weights of \cite[(2.5.1)]{BW-coloured} by setting $q=s^{-1/2}$ and $x=q^{-1/2}y$. 
		
		The fact that \eqref{eq:transfer-matrix}-\eqref{eq:vertex-weights} indeed define a stochastic matrix follows from the following observations. First, the vertex weights preserve the number of nontrivial markings -- they are nonzero only if the number of nontrivial entries in the bottom and right markings is equal to that in the top and left ones. 
		Second, the vertex weights are \emph{stochastic} -- with fixed markings of the bottom and right incident edges, the sum of the vertex weights over all possible assignments of the top and left edges is always equal to 1. This is not difficult to check directly, see also the proof of \cite[Proposition 2.5.1]{BW-coloured}. Third, the fact that in transition probabilities \eqref{eq:transfer-matrix} with a fixed bottom row, nonzero markings cannot run away infinitely far to the left with nontrivial probability is guaranteed by the condition that $y>1$, $0<q<1$, since then the top middle weight in \eqref{eq:vertex-weights} with $\I=(0,\dots,0)$ is strictly less than 1. Finally, if $\nu$ is a vector in $\Z^n$ with unequal coordinates, then each $\I^k(\nu)$ at the bottom of \eqref{eq:transfer-matrix} has no more than a single nonzero entry, and, thanks to the vanishing of the left bottom weight in \eqref{eq:vertex-weights} with $\I_{[1,n]}=1$, the same will be true for $\I^k(\mu)$, $k\in \Z$. Hence, for such $\nu$ the transition probabilities \eqref{eq:transfer-matrix} are nontrivial only if the vector $\mu$ also has unequal coordinates. The vertex weights in this case are written out explicitly in \cite[(3.32)]{deGier-Mead-Wheeler}, and they are manifestly nonnegative.    
		
		Let $\nu,\mu\in\Z^n$ be vectors with unequal coordinates. Denote by $T_{y,q}(\nu\to\mu)$ the transition probability afforded by  \eqref{eq:transfer-matrix}-\eqref{eq:vertex-weights}. Then, as is explained in \cite[Section 3.4.1]{deGier-Mead-Wheeler}, setting $y=1+(1-q)\varepsilon$ and taking the limit $\varepsilon\to 0$ yields the transition probabilities for the colored ASEP. More exactly, 
		\begin{equation}
			\label{eq:limit}
			\lim_{\epsilon\to 0} \left(T_{1+(1-q)\varepsilon,q}(\nu\to (\mu-([t/\epsilon],\dots,[t/\varepsilon]))\right)^{[t/\varepsilon]},
		\end{equation}
		converges to the time $t$ transition matrix for the colored ASEP, or TASEP at $q=0$. 
		
		Let us now use the fact that \cite[(9.5.2)]{BW-coloured} provides an integral representation for $(T_{y,q}(\nu\to\mu))^p$ for any $p\in\Z_{\ge 0}$, see also \cite[Theorem 3.4]{deGier-Mead-Wheeler}.
		It reads, substituting $s=q^{-1/2}$, $y_j=q^{-1/2} y$ there, and with other notations explained below,
		
		\begin{equation}
			\label{eq:Gmunu}
			\begin{gathered}
				(T_{y,q}(\nu\to\mu))^p	
				=	\frac{(-q)^{\frac 12(|\nu|-|\mu|)}q^{-np}}{(2\pi i)^n}  \oint_{C_1}
				\frac{dx_1}{x_1}
				\cdots 
				\oint_{C_n}
				\frac{dx_n}{x_n}
				\prod_{1 \leq i<j \leq n}
				\frac{x_j-x_i}{x_j-q x_i} \\
				\times
				\prod_{i=1}^{n}
				\left(\frac{x_i - q^{\frac 12} y}{x_i - q^{-\frac 12} y}\right)^p
				f_{\mu}({x}_1^{-1},\dots,{x}_n^{-1})
				g^{*}_{\nu}(x_1,\dots,x_n),
			\end{gathered}
		\end{equation}
		where $|\mu|=\sum_{j=1}^n \mu_j$, $|\nu|=\sum_{j=1}^n\nu_j$, and the integration contours satisfy \cite[Definition 8.1.1]{BW-coloured}:
		\begin{itemize}
			\item The contours $\{C_1,\dots,C_n\}$ are closed, positively oriented and pairwise non-intersecting.
			\item The contours $C_i$ and $q \cdot C_i$ are both contained within contour $C_{i+1}$ for all $1 \leq i \leq n-1$, where $q \cdot C_i$ denotes the image of $C_i$ under multiplication by $q$.
			\item All contours surround the point $s=q^{-1/2}$ and no other poles.
		\end{itemize} 
		
		The rational functions $f_\mu$ are defined by requiring that for $\delta=(\delta_1<\dots<\delta_n)\in \Z^n$,  
		\begin{equation}\label{eq:f-antidominant}
			f_\delta (x_1,\dots,x_n)= \prod_{i=1}^n \frac {1-q^{-1}}{1-q^{-\frac 12}x_i}\left(\frac{x_i-q^{-\frac 12}}{1-q^{-\frac 12}x_i}\right)^{\delta_i}, 
		\end{equation}
		and by the following recursion relation: If $\mu=(\mu_1,\dots,\mu_n)$ satisfies $\mu_i<\mu_{i+1}$ for some $1\le i\le n-1$, then 
		\begin{equation}
			\label{eq:f-exchange}
			\left(q - \frac{x_i-q x_{i+1}}{x_i-x_{i+1}} (1-\mathfrak{s}_i)\right) \cdot f_{\mu}(x_1,\dots,x_n) = f_{(\mu_1,\dots,\mu_{i+1},\mu_i,\dots,\mu_n)}(x_1,\dots,x_n).
		\end{equation}
		
		The correctness of this definition follows either from an independent definition of $f_\mu$ through suitable partition functions \cite[Definition 3.4.3]{BW-coloured}, or from commutation relations satisfied by the \emph{Demazure-Lusztig operators} in the left-hand side of \eqref{eq:f-exchange}.
		
		The \emph{dual} rational functions $g_\nu^*$ can be defined in a similar fashion, cf. \cite[(5.9.17)]{BW-coloured}, \cite[Theorem 5.8.1]{BW-coloured}, by requiring that for $\epsilon=(\epsilon_1>\dots>\epsilon_n)\in\Z^n$, 
		\begin{equation}\label{eq:g-antidominant}
			g_\epsilon(x_1,\dots,x_n)=\prod_{i=1}^n \frac 1{1-q^{-\frac 12}x_i}\left(\frac{x_i-q^{-\frac 12}}{1-q^{-\frac 12}x_i}\right)^{\epsilon_i},
		\end{equation}
		and by the following recursion relation: If $\nu=(\nu_1,\dots,\nu_n)$ is such that $\nu_i>\nu_{i+1}$, then 
		\begin{equation}
			\label{eq:g-exchange}
			\left(1-\frac{x_{i+1}-qx_i}{x_{i+1}-x_i}(1-\mathfrak s_i)\right)\cdot g_\nu(x_1,\dots,x_n)=g_{(\nu_1,\dots,\nu_{i+1},\nu_i,\dots,\nu_n)}(x_1,\dots,x_n).
		\end{equation}
		
		The integral representation \eqref{eq:Gmunu} is suitable for the limit transition \eqref{eq:limit}. Following the steps in 
		the proof of \cite[Theorem 3.5]{deGier-Mead-Wheeler}, we obtain the following formula for the transition probability of the colored ASEP over time $t$ between the initial and final point configurations $\nu$ and $\mu$:
		\begin{equation}
			\label{eq:ASEP-transition}
			\begin{gathered}
				\frac{(-q)^{\frac 12(|\nu|-|\mu|)}(-1)^n}{(2\pi i)^n}
				\oint
				\frac{dx_1}{x_1}
				\cdots 
				\oint
				\frac{dx_n}{x_n}
				\prod_{1 \leq i<j \leq n}
				\frac{x_j-x_i}{x_j-q x_i}\\ \times
				\prod_{j=1}^n
				\exp\left(\frac{(1-q)^2x_jt}{(1-x_j)(1-qx_j)}\right)
				f_{\mu}\left(q^{-\frac 12}{x}_1^{-1},\dots,q^{-\frac 12}{x}_n^{-1}\right)
				g^{*}_{\nu}\left(q^{\frac 12}x_1,\dots,q^{\frac 12}x_n\right),
			\end{gathered}
		\end{equation}
		where all variables are integrated over small positively oriented loops around 1. 
		
		Substituting the arguments of $f_\mu$ from \eqref{eq:ASEP-transition} into \eqref{eq:f-antidominant} gives 
		$$
		(-q)^{\frac{|\delta|}2}\prod_{i=1}^n \frac {(1-q)x_i}{1-q x_i}\left(\frac{1-x_i}{1-qx_i}\right)^{\delta_i},
		$$
		and swaps $x_i$ and $x_{i+1}$ in the difference operator in the left-hand side of \eqref{eq:f-exchange}. Similarly, substituting the arguments of $g_\nu^*$ into \eqref{eq:g-antidominant} gives  
		$$
		(-q)^{-\frac{|\epsilon|}2}\prod_{i=1}^n \frac 1{1-x_i}\left(\frac{1-qx_i}{1-x_i}\right)^{\epsilon_i},
		$$
		and does not change the difference operators in the left-hand side of \eqref{eq:g-exchange}. Thus, the power of $(-q)$ in front of \eqref{eq:ASEP-transition} cancels out, and we are free to set $q=0$ to reach the colored TASEP transition probability we are after. 
		
		The final step is the change of variables $w_j=x_j/(1-x_j)$. A straightforward computation shows that this takes \eqref{eq:ASEP-transition} with $q=0$ to the right-hand side of \eqref{eq:integral}, with the recursive definition of $f_\mu$ giving rise to \eqref{eq:antidominant}-\eqref{eq:phi-exchange}, and the recursive definition of $g_\nu^*$ giving rise to \eqref{eq:dominant}-\eqref{eq:psi-exchange}. 
	\end{proof}

	\section{Connection to the colored stochastic vertex models}\label{sec:observables} 
	
	The goal of this section is to explain how the observables $\mathrm{Prob}\{\M_C(t)=\mu\}$ of Theorem \ref{thm:integral} and Corollary \ref{cor:integral} originated from those for colored stochastic vertex models introduced in \cite{BW-observables}. In fact, Corollary \ref{cor:integral} is a direct degeneration of \cite[Theorem 1.1]{BW-observables} or \cite[Corollary 8.4]{Bufetov-Korotkikh}, and this is how we encountered it for the first time. While the approach we took in Section \ref{sec:integral} allowed us to prove the more general Theorem \ref{thm:integral}, it did not explain how the key observables were discovered, and such an explanation is precisely the {\it raison d'\^{e}tre} for the present section.  
	
	The colored stochastic vertex models of \cite{BW-observables} have a rather straightforward degeneration to the colored ASEP. This degeneration is very similar to the one used in the proof of Proposition \ref{prop:TASEP-transition} above, and it is described in \cite[Section 12.3.2]{BW-coloured}. We will not repeat the degeneration procedure here and introduce the observables of \cite{BW-observables} directly in the setting of the colored ASEP. 
	
	Consider a special initial condition of the colored ASEP, where position $k$ is occupied by a particle of color $(-k)$ for all $k\in\Z$; note that all colors are distinct -- this is the so-called \emph{rainbow case}. It is often referred to as the \emph{step} or \emph{fully packed initial condition}. Next, using the terminology of \cite{BW-observables}, let us choose a \emph{colored composition} of length $n\ge 1$ denoted as $\varkappa=(\varkappa_1,\dots,\varkappa_n)\in\Z^n$, where each part $\varkappa_j$ is assigned a color $b_j\in\Z$, $j=1,\dots,n$. In our degenerate case of the rainbow ASEP, it only makes sense to consider $\varkappa$ with pairwise distinct coordinates, which is what we will assume. 
	
	Let us define colored height functions for the colored ASEP. For its state $\eta\in\mathfrak C$, a color cutoff $c\in\Z$, and a position $k\in\Z$, 
	we set 
	\begin{equation}
		\label{eq:height-function}
		H_c(k)=|\{l\ge k:\eta(l)=c\}|, \qquad H_{>c}(k)=|\{l\ge k:\eta(l)>c\}|,  \qquad H_{\ge c}(k)=|\{l\ge k:\eta(l)\ge c\}|.
	\end{equation}
	For the random $\eta$ coming from the time evolution of the step initial condition, these quantities are integer-valued random variables. In particular, they take finite values almost surely. 
	
	Let us also define similar (although deterministic) quantities for the colored composition $\varkappa$:
	\begin{align*}
		H^\varkappa_c(k)&=\left|\{j\in\{1,\dots,n\}:\varkappa_j\ge k,\ b_j=c\}\right|, \\ H^\varkappa_{>c}(k)&=\left|\{j\in\{1,\dots,n\}:\varkappa_j\ge k,\ b_j>c\}\right|,\\
		H^\varkappa_{\ge c}(k)&=\left|\{j\in\{1,\dots,n\}:\varkappa_j\ge k,\ b_j\ge c\}\right|.
	\end{align*}
	
	The observables in question are labeled by colored compositions $\varkappa$ as above, and have the form, cf. \cite[Section 6.4]{BW-observables}:
	\begin{align*}
		\O_\varkappa (\eta) &= 
		\prod_{i=1}^n 
		\frac
		{q^{H_{\ge b_i}(\varkappa_i+1)-H_{\ge b_i}^\varkappa(\varkappa_i+1)}-q^{H_{>b_i}(\varkappa_i+1)-H_{>b_i}^\varkappa(\varkappa_i+1)}}{q-1}
		\\&=
		\prod_{i=1}^n \left(\mathbf{1}_{\{H_{b_i}(\varkappa_i+1)=1\}} \cdot q^{H_{>b_i}(\varkappa_i+1)-H_{>b_i}^\varkappa(\varkappa_i+1)}\right).  
	\end{align*}
	
	Here $q$ is the left jump rate of the colored ASEP. 
	The equality of the two expressions for $\O_\varkappa$ follows from the fact that for any $c\in\Z$, $H_c$ can only take values 0 or 1 (as we are in the rainbow case), as well as $H_{\ge c}=H_c+H_{>c}$, and $H_{b_i}^\varkappa(\varkappa_i+1)=0$. We will use the second expression. 
	
	In order to re-interpret this expression, we will use a remarkable \emph{color-position symmetry} that the colored ASEP with the step initial condition enjoys. It was originally discovered and proved in \cite{AAV}, see also \cite{BW-coloured}, \cite{BB}, \cite{Buf}. 
	
	Let $\eta(t)$ be the random state of the colored ASEP started from the step initial condition. Note that $\eta_t:\Z\to\Z$ is a bijection for every $t\ge 0$ with probability 1; let $\eta^{-1}_t:\Z\to\Z$ denote its inverse. Then $\eta_t^{-1}(c)$ is the location of the particle of color $c$ at time $t$ for every $c\in\Z$. The color-position symmetry says that $\eta_t$ and $\xi_t:=\eta_t^{-1}$ are \emph{equi-distributed} for any fixed $t\ge 0$.   
	
	Let us use the notation \eqref{eq:height-function} for the colored height functions associated with $\eta_t$, and a similar notation with $\widetilde H$ for the same quantities associated with $\xi_t=\eta_t^{-1}$. Then
	
	\begin{itemize}
		\item $H_i(j+1)=1$ is equivalent to $\eta_t^{-1}(i)\ge j+1$, or $\xi_t(i)\ge j+1$;   
		\item $H_{>i}(j+1)=h$ is equivalent to $|\eta_t(\{j+1,j+2,\dots\})\cap \{i+1,i+2,\dots\}|=h$, which is the same as 
		$|\{j+1,j+2,\dots\}\cap \eta^{-1}_t(\{i+1,i+2,\dots\})|=h$, or $\widetilde H_{>j}(i+1)=h$;
		
	\end{itemize}
	
	Hence,
	$$
	\O_\varkappa(\eta_t)=\prod_{i=1}^n \left(\mathbf{1}_{\{H_{b_i}(\varkappa_i+1)=1\}} \cdot q^{H_{>b_i}(\varkappa_i+1)-H^\varkappa_{>b_i}(\varkappa_i+1)}\right)=
	\prod_{i=1}^n \left(\mathbf{1}_{\{\xi_t(b_i)\ge \varkappa_i+1\}} \cdot q^{\widetilde H_{>\varkappa_i}(b_i+1)-H^\varkappa_{>b_i}(\varkappa_i+1)}\right).
	$$
	
	Let us now verify that if all the conditions in the indicator functions $\mathbf{1}_{\{\cdots\}}$ are satisfied, then all  exponents of $q$ in the above products are nonnegative. Indeed, if $H_{>b_i}^\varkappa(\varkappa_i+1)=l$, then there exist exactly $l$ indices $j_1,\dots,j_l\in\{1,\dots,n\}$ such that $\varkappa_{j_a}>\varkappa_i$ and $b_{j_a}>b_i$ for all $a=1,\dots,l$. The indicators in the right-hand side then tell us that $\xi_t(b_{j_a})\ge \varkappa_{j_a}+1>\varkappa_i+1$ for all $1\le a\le l$. Hence, $\widetilde H_{>\varkappa_i}(b_i+1)\ge l$, and the exponent of $q$ in the $i$th factor on the right-hand side is nonnegative. 
	
	Thus, we can now set $q$ to 0. This turns the colored ASEP into the colored TASEP, and $\O_\varkappa(\eta_t)$ can be written as
	\begin{equation}
		\label{eq:observable-final}
		\O_\varkappa(\eta_t)=\prod_{i=1}^n \left(\mathbf{1}_{\{\xi_t(b_i)\ge \varkappa_i+1\}}\cdot  \mathbf{1}_{\left\{\widetilde H_{>\varkappa_i}(b_i+1)=H^\varkappa_{>b_i}(\varkappa_i+1)\right\}}\right).
	\end{equation}     
	
	Let $C=(c_1<\dots<c_n)\in\Z^n$ be such that $\{c_1,\dots,c_n\}=\{\varkappa_1+1,\dots,\varkappa_n+1\}$, and let $\sigma$ be a permutation of $\{1\dots,n\}$ such that $c_j=\varkappa_{\sigma(j)}+1$ for all $1\le j\le n$. Then an examination of the 
	right-hand side of \eqref{eq:observable-final} shows that it is exactly equal to $\mathrm{Prob}\{\M_C(\xi_t)=(b_{\sigma(1)},\dots,b_{\sigma(n)})\}$, yielding 
	\begin{equation}
		\label{eq:equal-observables}
		\O_\varkappa(\eta_t)=\mathrm{Prob}\left\{\M_C(\xi_t)=\left(b_{\sigma(1)},\dots,b_{\sigma(n)}\right)\right\}.
	\end{equation}
	
	Thus, we see that the integral representation for $\mathbb{E} [\O_\varkappa]$ that was obtained in \cite{BW-observables} can be used to produce an integral representation for the right-hand side of \eqref{eq:equal-observables}. This would result in an alternative proof of Corollary \ref{cor:integral} in the case of the step initial condition, but would not go as far as proving the more general Theorem \ref{thm:integral}. The duality seen in Proposition \ref{prop:duality} was also anticipated (but not proved) in \cite[Section 6.3]{BW-observables}.  
	
	\section{Observables for leaders}
	
	In this section we introduce our main quantity of interest --- the notion of \textit{leaders} in multi-type TASEP --- and relate its behavior to the observables $\M_C (t)$ given by Corollary \ref{cor:integral}. 
	
	\subsection{Leaders: Definition}
	
	In this and subsequent sections we consider TASEP which starts at time $0$ from a configuration $\eta_0 (x)=-x$, for all $x \in \Z$. As before, we denote by $\eta_t$ the configuration of this TASEP at time $t$.
	Note that we have only one particle of each type, and the configuration is given by a permutation of $\Z$ at any time. The equality $\eta^{-1}_t (a) =b$ means that (the unique) particle of type $a$ is at position $b$ at time $t$; while the equality $\eta_t (a) =b$ means, as before, that (the unique) particle of type $b$ is at position $a$ at time $t$, for any $a,b \in \Z$. 
	
	In this paper, we mostly focus on the asymptotic behavior of \textit{leaders}, a notion that we now introduce. 
	We call a particle of type $i$ (for $i \ge k$) \textit{a $k$-leader} at time $t$ if $\eta^{-1}_t  (i) \ge \eta^{-1}_t  (z)$, for all $z \ge k$. More generally, we say that a particle of color $i$ is \textit{a $(k,s)$-leader} if $\eta^{-1}_t (i)$ is the $s$-th largest number among $\{ \eta^{-1}_t (z) \}_{z \ge k}$. Clearly, for any $s$ and $k$ there is a uniquely defined $(s,k)$-leader for any fixed configuration. We will be interested in the distribution of types of these leaders. 
	
	\begin{ex}
		
		\begin{center}
			\setlength{\tabcolsep}{4pt} 
			\begin{tabular}{l *{14}{>{\small}c}}
				\hline
				\textbf{types of particles} & ... & >4 & >4 & 0  & 4 & -1 & 3 & -2 & 1 & 2 & <-2 & <-2 & ... & \\
				\textbf{positions} &... & ... & ... & -4 & -3 & -2 & -1 & 0 & 1 & 2 & ... & ... & ... &   \\
				\hline
			\end{tabular}
		\end{center}
		
		\medskip
		
		Consider the configuration above, 
		where the particle of type $0$ stands at position $-4$, the particle of type $4$ stands at position $-3$, and so on, while all positions $\le -5$ are occupied by all particles of types $>4$, and all positions $\ge 3$ are occupied by all particles of types $<-2$.
		
		For such a configuration, the $0$-leader is the particle of type $2$, the $(0,2)$-leader is the particle of type $1$, the $1$-leader is the particle of type $2$ again, and the $3$-leader is the particle of type $3$. 
	\end{ex} 
	
	\subsection{Leader distribution}
	
	\begin{prop}
		\label{prop:1-leader-formula}
		For any $k_2 > k_1 \in \Z$, we have
		\begin{equation}
			\label{eq:LeaderAsymp1v2}
			\mathrm{Prob} \left( \mbox{$k_1$-leader is of type $\ge k_2$} \right) = \frac{1}{2 \pi \ii} \oint_{|w|=R>1}
			\frac{w+2}{w} (1+w)^{k_1-k_2-1} \exp \left( \frac{t w^2}{1+w} \right) dw,
		\end{equation}
		where the integration contour is positively oriented.
	\end{prop}
	
	\begin{proof}
		
		Let $C=(k_1, k_2)$, for $k_2>k_1$, and $\mu = (x_1,x_2)$, with $x_1<x_2$. In that case, initially $\M_C(0)= (-k_1, -k_2)$. 
		
		By Corollary \ref{cor:integral}, we have
		\begin{equation}\label{eq:obs2-1-2var}
			\begin{gathered}
				\mathrm{Prob}\{\M_C(t)=\mu\}=\frac{1}{(2\pi \ii)^2} \oint_{|w_1|=R>1} \oint_{|w_2|=R>1} \left(1-\frac{w_1}{w_2}\right) \\ \times (1+w_1)^{-x_1} (1+w_2)^{-x_2} 
				(1+w_1)^{-k_1-2} (1+w_2)^{-k_2-1} \exp (t (w_1+w_2)) d w_1 d w_2 \,.
			\end{gathered}
		\end{equation}
		The event $\M_C(t)=\mu$ with such a choice of $C$ and $\mu$ can be rephrased as the event that the $k_1$-leader is at position $x_2$ and has type $\ge k_2$, while the $(k_1,2)$-leader is at position $x_1<x_2$. 
		
		We need to sum these probabilities over all possible values of $x_2$ and $x_1$ in order to exclude unnecessary (for our current purpose) information about locations. Doing the summations
		\begin{equation}
			\label{eq:obs2-1-3var}
			\sum_{x_1=-\infty}^{x_2-1} (1+w_1)^{-x_1} = \frac{(1+w_1)^{-x_2+1}}{-w_1}, \qquad |1+w_1| <1,
		\end{equation}
		and
		\begin{equation}
			\label{eq:obs2-1-4var}
			\sum_{x_2=-k_1}^{\infty} (1+w_1)^{-x_2} (1+w_2)^{-x_2} = \frac{(1+w_1)^{k_1+1} (1+w_2)^{k_1+1}}{w_1+w_2+w_1 w_2}, \qquad |(1+w_1)(1+w_2)| >1,
		\end{equation}
		in \eqref{eq:obs2-1-2var}, we obtain
		\begin{multline}
			\label{eq:LeaderAsymp-Type2}
			\mathrm{Prob} \left( \mbox{$k_1$-leader is of type $\ge k_2$} \right) \\ = \frac{1}{(2 \pi \ii)^2} \oint_{w_1} \oint_{w_2}
			\frac{w_2-w_1}{w_1 w_2} \frac{(1+w_2)^{k_1-k_2}}{w_1+w_2+w_1 w_2} \exp (t (w_1+w_2)) dw_1 dw_2,
		\end{multline}
		where the contour of integration in $w_1$ is a small circle around $(-1)$, and the contour of integration in $w_2$ is a large circle around $(-1)$ such that $|(1+w_1)(1+w_2)| >1$ (we first choose the radius for the $w_1$-contour, and then for the $w_2$-contour). The integrand has a unique $w_1$-pole at $w_1 = -w_2 (1+w_2)^{-1}$, which is inside the $w_1$ contour due to our conditions on contours. Taking the residue, we arrive at the statement of the proposition.
		
	\end{proof}
	\begin{prop}
		\label{prop:leader-position}
		For any integers $k_2 > k_1$ and $x \ge k_1$, we have
		\begin{multline}
			\label{eq:LeaderAsymp1v2}
			\mathrm{Prob} \left( \mbox{$k_1$-leader is of type $\ge k_2$ and is located at $x$} \right) \\ = \frac{1}{2 \pi \ii} \oint_{w_1} \oint_{w_2}
			\frac{w_1-w_2}{w_1 w_2} (1+w_1)^{-k_1-x-1} (1+w_2)^{-k_2-x-1} \exp \left( t (w_1 +w_2) \right) dw_1 dw_2,
		\end{multline}
		where the $w_1$-contour encircles $-1$ and no other poles of the integrand, while $w_2$-contour encircles the poles at $w_2=-1$ and $w_2=0$, but no other poles of the integrand. Both contours are positively oriented.
	\end{prop}
	
	\begin{proof}
		
		The proof is exactly the same as in Proposition \ref{prop:1-leader-formula} with the only difference that we perform the summation \eqref{eq:obs2-1-3var}, but not \eqref{eq:obs2-1-4var}, and plug the result into \eqref{eq:obs2-1-2var}.
		
	\end{proof}
	
	\subsection{Joint distribution of two leaders}
	
	\begin{prop}
		\label{prop:2leaders-1}
		Let $k_3 >k_2 > k_1$ be arbitrary integers. We have the following equality:
		\begin{multline}
			\label{eq:prop2leaders1}
			\mathrm{Prob} \left( \mbox{the $k_1$-leader has type $\ge k_3$, } \mbox{the $(k_1, 2)$-leader has type $\ge k_2$} \right)
			\\ =
			\frac{1}{(2 \pi \ii)^2} \oint_{|v_2|=2} \oint_{|v_3|=2} \frac{(v_3-v_2)(v_2 v_3^2-1)(v_2^2 v_3-1) v_2^{k_1-k_2} v_3^{k_1-k_3}}{ v_2^2 v_3 (v_2 v_3-1) (v_2-1) (v_3-1)^3 } \\ \times \exp \left( t \left( \frac{1}{v_2 v_3} + v_2 + v_3 - 3 \right) \right) d v_2 d v_3.
		\end{multline}
		where the integration contours are counterclockwise.
	\end{prop}
	
	\begin{proof}
		
		Let $C=(k_1, k_2,k_3)$, for $k_3>k_2>k_1$, and $\mu = (x_1,x_2,x_3)$, with $x_1<x_2<x_3$. In that case, $\M_C(0)= (-k_1, -k_2,-k_3)$. By Corollary \ref{cor:integral}, we have
		\begin{equation}\label{eq:obs2-2-2var}
			\begin{gathered}
				\mathrm{Prob}\{\M_C(t)=\mu\}=\frac{1}{(2\pi \ii)^3} \oint_{|w_1|=R>1} \oint_{|w_2|=R>1} \oint_{|w_3|=R>1} \frac{(w_3-w_1)(w_3-w_2)(w_2-w_1)}{w_2 w_3^2} \\ \times (1+w_1)^{-x_1-k_1-3} (1+w_2)^{-x_2-k_2-2} (1+w_3)^{-x_3-k_3-1} 
				{e^{t(w_1+w_2 + w_3)} d w_1 d w_2 d w_3}\,.
			\end{gathered}
		\end{equation}
		The event $\M_C(t)=\mu$ with such a choice of $C$ and $\mu$ can be rephrased as the event that $k_1$-leader is at position $x_3$ and has type $\ge k_3$, the $(k_1,2)$-leader is at position $x_2$ and has type $\ge k_2$, and the $(k_1,3)$-leader is at position $x_1$. 
		
		We are interested in the event
		\begin{equation}
			\label{eq:asymp2joint1}
			\mathrm{Prob} \left( \mbox{the $k_1$-leader has type $\ge k_3$, } \mbox{the $(k_1, 2)$-leader has type $\ge k_2$} \right).
		\end{equation}
		In order to obtain it, we need to sum over $x_i$'s. 
		
		First, let us deform the contours to positively oriented circles $|w_1+1|=\varepsilon$, $|w_2|=2$, $|w_3|=2/\varepsilon$, for a positive $\varepsilon \ll 1$. Second, we perform the following summations:
		$$
		\sum_{x_1=-\infty}^{x_2-1} (1+w_1)^{-x_1-3} = \frac{(1+w_1)^{-x_2-2}}{-w_1}, \qquad \left| 1 + w_1 \right| <1.
		$$
		$$
		\sum_{x_2=-\infty}^{x_3-1} (1+w_1)^{-x_2-2} (1+w_2)^{-x_2-2} = \frac{(1+w_1)^{-x_3-1} (1+w_2)^{-x_3-1} }{1-(1+w_1)(1+w_2)}, \qquad \left| (1 + w_1)(1+w_2) \right| <1.
		$$
		\begin{multline*}
			\sum_{x_3=-k_1}^{+\infty} (1+w_1)^{-x_3-1} (1+w_2)^{-x_3-1} (1+w_3)^{-x_3-1} = \frac{(1+w_1)^{k_1-1} (1+w_2)^{k_1-1} (1+w_3)^{k_1-1} }{1-\left( (1+w_1)(1+w_2)(1+w_3) \right)^{-1}}
			\\ = \frac{(1+w_1)^{k_1} (1+w_2)^{k_1} (1+w_3)^{k_1} }{(1+w_1)(1+w_2)(1+w_3) -1}, \qquad \left| (1 + w_1)(1+w_2)(1+w_3) \right| >1.
		\end{multline*}
		Note that for our chosen contours the indicated inequalities on $w_1,w_2, w_3$ hold. Plugging the results of these summations sequentially into \eqref{eq:obs2-2-2var}, we arrive at the equality
		\begin{multline*}
			\mathrm{Prob} \left( \mbox{the $k_1$-leader has type $\ge k_3$,} \mbox{the $(k_1, 2)$-leader has type $\ge k_2$} \right)
			\\ =
			\frac{1}{(2 \pi \ii)^3} \oint_{|w_1+1|=\varepsilon} \oint_{ |w_2|=2} \oint_{|w_3|=2/\varepsilon} \frac{(w_3-w_2)(w_3-w_1)(w_2-w_1) (1+w_2)^{k_1-k_2} (1+w_3)^{k_1-k_3}}{w_1 w_2 w_3^2 ( (1+w_1)(1+w_2)-1 ) ( (1+w_1)(1+w_2)(1+w_3) -1 ) } \\ \times \exp \left( t(w_1+w_2+w_3) \right) d w_1 dw_2 dw_3.
		\end{multline*}
		
		Next, let us make a simple change of variables $1+w_i=v_i$ in the integral expression above. We obtain
		\begin{multline*}
			\frac{1}{(2 \pi \ii)^3} \oint_{|v_1|=\varepsilon} \oint_{ |v_2-1|=2} \oint_{|v_3-1|=2/\varepsilon} \frac{(v_3-v_2)(v_3-v_1)(v_2-v_1) v_2^{k_1-k_2} v_3^{k_1-k_3}}{(v_1-1) (v_2-1) (v_3-1)^2 ( v_1 v_2 -1 ) ( v_1 v_2 v_3 -1 ) } \\ \times \exp \left( t(v_1+v_2+v_3 - 3) \right) d v_1 d v_2 d v_3.
		\end{multline*}
		
		For our contours, the only $v_1$-pole is inside the contour is $v_1 = 1/(v_2 v_3)$. Taking the residue in this $v_1$-pole (and also moving the contours after that), we arrive at the statement of the proposition. 
		
	\end{proof}
	
	\begin{prop}
		\label{prop:2leaders-2}
		Let $k_3 > k_2 >k_1$ be arbitrary integers. We have the following equality:
		\begin{multline*}
			\mathrm{Prob} \left( \mbox{the $k_1$-leader has type $\in [k_2;k_3)$,} \mbox{the $(k_1, 2)$-leader has type $\ge k_3$} \right) 
			\\ = \frac{1}{(2 \pi \ii)^2} \oint_{|v_2|=2} \oint_{|v_3|=2} \frac{(v_2 v_3^2-1)(v_2^2 v_3-1) \left( v_2^{k_1-k_2} v_3^{k_1-k_3} - v_2^{k_1-k_3} v_3^{k_1-k_2} \right) }{ v_2^2 v_3 (v_2 v_3-1) (v_2-1) (v_3-1)^2 } \\ \times \exp \left( t \left( \frac{1}{v_2 v_3} + v_2 + v_3 - 3 \right) \right) d v_2 d v_3.
		\end{multline*}
		
	\end{prop}
	
	\begin{proof}
		Let $C=(k_1, k_2,k_3)$, for $k_3>k_2>k_1$, and $\mu = (x_1,x_3,x_2)$, with $x_1<x_2<x_3$. In that case, again, $\M_C(0)= (-k_1, -k_2,-k_3)$. 
		
		By Corollary \ref{cor:integral}, we have (applying \eqref{eq:phi-exchange} once)
		\begin{multline}
			\label{eq:2leader-2obs}
			\mathrm{Prob}\{\M_C(t)=\mu\} =
			\frac{1}{(2 \pi \ii)^3} \oint \oint \oint \frac{(w_3-w_2)(w_3-w_1)(w_2-w_1)}{w_2 w_3^2} (1+w_1)^{-x_1} \\ \times \frac{\left( (1+w_2)^{-x_2} (1+w_3)^{-x_3} - (1+w_2)^{-x_3} (1+w_3)^{-x_2} \right) w_3 (1+w_2)}{w_2-w_3} \\ \times (1+w_1)^{-k_1-3} (1+w_2)^{-k_2-2} (1+w_3)^{-k_3-1} \exp \left( t(w_1+w_2+w_3) \right) d w_1 dw_2 dw_3.
		\end{multline}
		
		The event $\M_C(t)=\mu$ with such a choice of $C$ and $\mu$ can be rephrased as the event that the $k_1$-leader is at position $x_3$ and has type $\ge k_2$ but less than $k_3$, the $(k_1,2)$-leader is at position $x_2$ and has type $\ge k_3$, and the $(k_1,3)$-leader is at position $x_1$.
		In order to get the probability that we are interested in, we again need to sum this expression over the possible values of $x_1$, $x_2$, and $x_3$.
		
		Note that the integral in \eqref{eq:2leader-2obs} can be written as a difference of two terms, the second of these terms is obtained from the first one by swapping $x_2$ and $x_3$. Moreover, note that the integrand in \eqref{eq:2leader-2obs} is actually symmetric in $w_2$ and $w_3$. Therefore, the second term can be obtained from the first one also by swapping $k_2$ and $k_3$ (rather than $x_2$ and $x_3$), which is technically easier since we need to sum over $x_1$, $x_2$, and $x_3$. 
		
		Thus, we will perform the summations in the first of these terms and then use the symmetry to obtain the second one. However, the first term is very similar to \eqref{eq:obs2-2-2var}, and the summations in it can be done via exactly the same summation formulas and contour deformations as in Proposition \ref{prop:2leaders-1}. We obtain  
		
		\begin{multline*}
			\frac{1}{(2 \pi \ii)^2} \oint_{|v_2|=2} \oint_{|v_3|=2} \frac{(v_2 v_3^2-1)(v_2^2 v_3-1) v_2^{k_1-k_2} v_3^{k_1-k_3}}{ v_2^2 v_3 (v_2 v_3-1) (v_2-1) (v_3-1)^2 }  \exp \left( t \left( \frac{1}{v_2 v_3} + v_2 + v_3 - 3 \right) \right) d v_2 d v_3.
		\end{multline*}
		
		Subtracting the second term from this expression, we arrive at the statement of the proposition. 
		
	\end{proof}
	
	\begin{prop}
		\label{prop:number-changes-formula}
		The probability of the event that at a fixed time $t \in \R_{>0}$, there exists $x \in \Z$ such that the $0$-leader is at position $x$, the $(0,2)$-leader is at position $x+1$, and the type of the $(0,2)$-leader is larger than the type of the $0$-leader, is equal to 
		\begin{multline}
			\label{eq:number-changes-formula}
			\frac{1}{(2 \pi \ii)^2} \oint_{|v_2|=2} \oint_{|v_3|=2} \frac{(v_2-v_3)(v_2 v_3^2-1)(v_2^2 v_3-1)}{v_2 v_3^2 (v_2-1) (v_3-1)(v_2 v_3-1)^2}
			\exp \left( t \left( \frac{1}{v_2 v_3}+v_2+v_3 - 3 \right) \right) dv_2 dv_3.
		\end{multline}
	\end{prop}
	
	\begin{proof}
		We start with the formula \eqref{eq:2leader-2obs}, which deals with the case when the $(0,2)$-leader is of larger type than the $0$-leader. In \eqref{eq:2leader-2obs}, we substitute $k_1=0$, $x_2=x_3-1$, and then perform summations, which are very similar to the previous cases (and therefore omitted). 
	\end{proof}

	\section{Large time asymptotics}
	
	\subsection{Type of leader}
	\label{sec:5-1}
	
	\begin{thm}
		\label{th:leader-type}
		Let $k \in \Z$ and $L_1 (t)$ be the type of the $k$-leader in the multi-type TASEP at time $t$. Then
		$$
		\frac{L_1(t) - k}{\sqrt{t}} \xrightarrow[t \to \infty]{} \frac{1}{\sqrt{\pi}} \exp \left( \frac{-a^2}{4} \right) \mathbf{1}_{a \ge 0} \, da, \qquad \mbox{in distribution}.
		$$
	\end{thm}
	\begin{proof}
		
		In order to obtain the asymptotics of integrals, we apply (here and also in all subsequent sections) the well-known method of the steepest descent. We refer to, e.g., \cite{deBruijn-AMA} for the basics of this method.
		
		We start from the formula in the statement of Proposition \ref{prop:1-leader-formula}. Let us consider the sequence $k_2 = k_2 (t)$ such that $k_1 - k_2 -1 = - \lfloor a \sqrt{t} \rfloor$, for $a>0$. In such a case, the part of the integrand in \eqref{eq:LeaderAsymp1v2} that depends on $t$ can be written as
		\begin{equation}
			\label{eq:descentFunc-1}
			\exp \left( t \left( \frac{w_2^2}{1+w_2} - \frac{a}{\sqrt{t}} \ln (1+w_2) \right) \right).
		\end{equation}
		The function in the exponent has a unique saddle point at $w_2=0$; we deform the contour of integration as shown in Figure \ref{fig:steepest-1}. 
		
		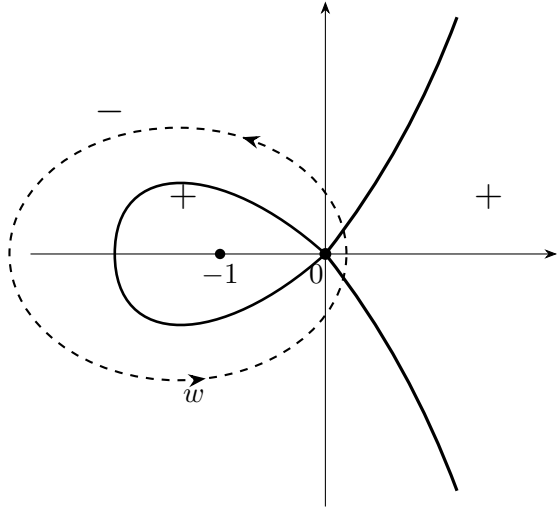
\begin{figure}[ht]
			\centering
			\begin{tikzpicture}[scale=1.4,>=Stealth]
				
				\draw[->] (-2.8,0) -- (2.2,0);
				\draw[->] (0,-2.4) -- (0,2.4);
				
				\fill (0,0) circle (1.6pt);
				\node[below left] at (0.08,0) {$0$};
				
				\draw[line width=1.0pt]
				(0,0) .. controls (-1,0.9) and (-2,0.9) .. (-2,0)
				.. controls (-2,-0.9) and (-1,-0.9) .. (0,0);
				
				\fill (-1,0) circle (1.4pt);
				\node[below] at (-1,0) {$-1$};
				
				\draw[dashed,line width=0.8pt,
				postaction={decorate},
				decoration={markings,
					mark=at position 0.18 with {\arrow{>}},
					mark=at position 0.78 with {\arrow{>}}}]
				(-1.4,0) ellipse (1.6 and 1.2);
				
				\node at (-1.35,0.55) {\Large $+$};
				\node at (-2.05,1.35) {\Large $-$};
				\node at (1.55,0.55) {\Large $+$};
				\node at (-1.25,-1.35) {$w$};
				
				\draw[line width=1.2pt]
				(0,0) .. controls (0.55,0.7) and (0.95,1.45) .. (1.25,2.25);
				\draw[line width=1.2pt]
				(0,0) .. controls (0.55,-0.7) and (0.95,-1.45) .. (1.25,-2.25);
				
			\end{tikzpicture}
			\caption{Solid line: $\Re \left( \frac{w^2}{1+w} \right)=0$, the regions of positive and negative real part of this function are also depicted. We deform (without encountering any poles) the contour of integration to the dashed line, which lies almost entirely in the region with negarive real part, and only in an (infinitesimal, of order $t^{-1/2}$) neighborhood of $w=0$ the values of $\Re \left( \frac{w^2}{1+w} \right)$ on the contour are close to 0. }
			\label{fig:steepest-1}
		\end{figure}
		
		After that, we make a change of variables $w_2 = t^{-1/2} u$, and use
		the Taylor expansion
		$$
		\frac{t w_2^2}{1+w_2} - \sqrt{t} a \log \left( 1 + w_2 \right) = u^2 - au + O \left( t^{-1/2} \right).
		$$
		After such a deformation of contours, by standard for the steepest descent method estimates, we obtain that the asymptotics of the integral is dominated by a neighborhood of the saddle point, and it is given by 
		
		\begin{multline}
			\label{eq:1-LeaderAsymp-v3}
			\mathrm{Prob} \left( \mbox{$(k_1+1)$-leader is of type $\ge k_1 + \lfloor a \sqrt{t} \rfloor$} \right) \\ = \frac{1}{2 \pi \ii} \int_{\Im (w_2) =  \eps}
			\frac{2}{u} \exp \left( u^2 - au \right) du + O \left( t^{-1/2} \right) = 1 - \mathrm{erf} \left( \frac{a}{2} \right) + O \left( t^{-1/2} \right),
		\end{multline}
		where $0 < \eps \ll 1$, and $\mathrm{erf}$ is the Gauss error function. Note that \eqref{eq:1-LeaderAsymp-v3} varies from 1 to 0 as $a$ ranges from $0$ to $+\infty$, which shows that we chose the correct fluctuation scale. 
		
		Differentiating equality \eqref{eq:1-LeaderAsymp-v3} in $a$, we arrive at the statement of the theorem.
	\end{proof}
	
	\subsection{Type of leader conditioned on position}
	
	\begin{thm}
		\label{th:leader-conditioned-asymp}
		Let $k \in \Z$, $L_1 (t)$ be the type of the $k$-leader in the multi-type TASEP at time $t$, let $X_1 (t)$ be the position of the $k$-leader, and let $x \in \R$. We have
		\begin{multline}
			\lim_{t \to \infty} \mathrm{Law} \left( \frac{L_1 (t)-k}{\sqrt{t}} \Biggl| X_1 (t) = \lfloor t + x \sqrt{t} \rfloor \right) = \frac{1}{\sqrt{2 \pi}} \exp \left( \frac{-(x+a)^2}{2} \right) 
			\\ \times \left( 1 + \frac{\sqrt{\pi}}{2} \left( 1 + \mathrm{erf} \left( \frac{x}{\sqrt{2}} \right) \right) \cdot \exp \left( \frac{x^2}{2} \right) \cdot \left( a+x \right) \right) \mathbf{1}_{a \ge 0} \, da, \qquad \mbox{in distribution},
		\end{multline}
		where in the left-hand side we consider the conditional distribution of the type of the leader (note that according to the classical central limit theorem, $\lfloor t + x \sqrt{t} \rfloor$ is a typical value for the position of the leader). 
	\end{thm}
	
	\begin{proof}
		We analyze the formula given by Proposition \ref{prop:leader-position} with $k_1 :=k$, $k_2 := \lfloor a \sqrt{t} \rfloor$, and $x := x \sqrt{t}$, where in the left-hand sides we use the notation from Proposition \ref{prop:leader-position}, and in the right-hand sides -- the notation from the statement of the current proposition. We apply the steepest descent method, and deform both contours in the same way as in Section \ref{sec:5-1}. After that, the principal contribution to the integral is given by a neighborhood of the critical point $w_1=w_2=0$, and we change variables $\sqrt{t} w_i =:v_i$, for $i=1,2$, in order to compute it. Using the Taylor expansion, we obtain
		
		\begin{multline}
			\mathrm{Prob} \left( \mbox{$k$-leader is of type $\ge k +a \sqrt{t}$ and is located at $\lfloor t + x \sqrt{t} \rfloor$} \right) \\ = \frac{1}{\sqrt{t} \cdot 2 \pi \ii} \int_{\Im (v_1) = - \eps } d v_1 \int_{\Im (v_2)= \eps } d v_2 
			\left( \frac{1}{v_2} - \frac{1}{v_1} \right) \exp \left( \frac{v_1^2}{2} - \ii x v_1 + \frac{v_2^2}{2} - (x+a) v_2 \right) + O \left( \frac{1}{t} \right).
		\end{multline} 
		Making a change of variable $y_i = \ii v_i$, for $i=1,2$, we obtain that the expression above is asymptotically equivalent to 
		\begin{equation}
			\label{eq:asympt-lead-cond}
			\frac{1}{\sqrt{t} \cdot 2 \pi \ii} \int_{\R + \ii \eps } d y_1 \int_{\R - \ii \eps } d y_2 
			\left( \frac{1}{y_1} - \frac{1}{y_2} \right) \exp \left( \frac{-y_1^2}{2} - \ii x y_1 + \frac{-y_2^2}{2} - \ii (x+a) y_2 \right).
		\end{equation}
		In order to compute this integral, we use the following table integrals and Sokhotski-Plemelj theorem:
		$$
		\mathrm{p.v.} \int_{-\infty}^{+\infty} \frac{1}{s} \exp \left( \frac{-s^2}{2} + \ii s x \right) ds
		= \ii \cdot \pi \cdot \mathrm{erf}\left( \frac{x}{\sqrt{2}} \right),
		$$
		$$
		\int_{-\infty}^{+\infty} \exp \left( \frac{-y^2}{2} - \ii y x \right) dy = \sqrt{2 \pi} \exp \left( - \frac{x^2}{2} \right),
		$$
		$$
		\lim_{\eps \to 0} \int \frac{f(x)}{x \pm \ii \eps} dx = \mp \ii \pi f(0) + \mathrm{p.v.} \int_{-\infty}^{+\infty} \frac{f(x)}{x} dx. 
		$$
		Plugging in these expressions, we obtain that \eqref{eq:asympt-lead-cond} is equal to
		\begin{equation}
			\label{eq:dddd}
			\frac{1}{\sqrt{2 \pi t}} \exp \left( \frac{-x^2}{2} \right) \cdot \frac{1}{2} \left( 1 + \mathrm{erf} \left( \frac{-x-a}{\sqrt{2}} \right) + \left( 1 - \mathrm{erf} \left( \frac{-x}{\sqrt{2}} \right) \right) \exp \left( - ax - \frac{a^2}{2} \right) \right).
		\end{equation}
		By the classical Central Limit Theorem, the factor $\frac{1}{\sqrt{2 \pi t}} \exp \left( \frac{-x^2}{2} \right)$ is asymptotically equivalent to $\mathrm{Prob} \left( X_1 (t) = \lfloor t + x \sqrt{t} \rfloor \right)$. Therefore, the rest of \eqref{eq:dddd} corresponds to the conditional law of the type of the leader. Differentiating it with respect to $a$, we arrive at the statement of the proposition. 
	\end{proof}
	
	\begin{rmk}
		In Theorem \ref{th:leader-conditioned-asymp} we condition on a typical position $t + x \sqrt{t}$ of the leader. It is quite natural to consider the conditioning on atypical positions of the leader as well, namely, for $X_1 = \lfloor A t \rfloor$, with $A >1$ (atypically fast first particle in TASEP) and $A<1$ (atypically slow first particle in TASEP). In the former case, the type of the leader minus $k$ will have finite (i.e., not growing with $t$) value and it is possible to show that asymptotically, as $t \to \infty$, it will be geometric with rate $A^{-1}$. In the latter case, the type of the leader will grow as $(1-A)t$ and will have fluctuations of order $\sqrt{t}$. We do not provide the details of these calculations since they are very similar to the ones provided in this paper. 
	\end{rmk}

	\subsection{Joint distribution of two leaders}
	
	\begin{thm}
		\label{th:joint}
		Let us fix $k_1 \in \Z$. Let $L_1 (t), L_2 (t) $ be the types of the $(k_1,1)$- and $(k_1,2)$-leaders, respectively, at time $t$. We have the following convergence in distribution:
		\begin{equation}
			\lim_{t \to \infty} \left( \frac{L_1(t) - k_1}{\sqrt{t}}, \frac{L_2(t) - k_1}{\sqrt{t}} \right) = G (a_2, a_3) d a_2 d a_3,
		\end{equation}
		where the limiting density $G (a_2, a_3)$ is given via
		$$
		G (a_2, a_3) :=
		\begin{cases}
			d_1 (a_2, a_3), \qquad \mbox{for $a_2 > a_3$,} \\
			d_2 (a_2,a_3), \qquad \mbox{for $a_2 \le a_3$,}
		\end{cases}
		$$
		and
		\begin{multline}
			\label{eq:density2a}
			d_1 (a_2,a_3) :=
			\frac{-1}{16 \pi} \left( \left( 4 \cdot \mathrm{erf} \left( \frac{\sqrt{3}}{6} (a_2 +a_3) \right) -4 \right) \exp \left( \frac{(a_2-a_3)^2}{4} \right) \sqrt{\pi} (a_2 - a_3)
			\right. \\ \left. + 2 \sqrt{3} \left( a_2^2 -6 \right) \exp \left( \frac{-a_2^2+a_2 a_3 - a_3^2}{3} \right)
			\right. \\ \left. + \sqrt{\pi} \exp \left( \frac{-a_2^2}{4} \right) \left(  \mathrm{erf} \left( \frac{\sqrt{3}}{6} (a_2 - 2 a_3) \right) +1 \right) \left( - 4 a_2 +a_2^3 - 2 a_2^2 a_3 + 4 a_3 \right) \right. ,
		\end{multline}
		\begin{multline*}
			d_2 (a_2,a_3) := \frac{1}{4 \sqrt{\pi}} \left( - \exp \left( -\frac14 (a_2-a_3)^2 \right) (a_2 - a_3) \left( \mathrm{erf} \left( \frac{\sqrt{3}}{6} (a_2+a_3) \right) -1 \right)
			\right. \\ \left. + a_3 \exp \left( \frac{-a_3^2}{4} \right) \mathrm{erf} \left( \frac{\sqrt{3}}{6} (2 a_2-a_3) \right) +a_2 \exp \left( \frac{-a_2^2}{4} \right) \mathrm{erf} \left( \frac{\sqrt{3}}{6} ( a_2- 2 a_3) \right)
			\right. \\ \left. - a_3 \exp \left( \frac{-a_3^2}{4} \right) + a_2 \exp \left( \frac{-a_2^2}{4} \right) \right).
		\end{multline*}
		
	\end{thm}
	
	\begin{rmk}
		\label{rem:permut}
		By a direct integration, it follows from Theorem \ref{th:joint} that
		$$
		\lim_{t \to \infty} \mathrm{Prob} \left( L_1(t) > L_2 (t) \right) = 
		\lim_{t \to \infty} \mathrm{Prob} \left( L_1(t) < L_2 (t) \right) = \frac12. 
		$$
		We are unaware of a conceptual explanation or a shorter proof of this equality; it would be interesting to understand the origin of this fact better. A naive generalization for the ordering of types of more than two leaders --- that they always form a uniformly random permutation --- is false. 
	\end{rmk}
	
	\begin{proof}
		
		\textbf{First step}. 
		
		\smallskip
		
		Choose the sequences $k_2 = k_2(t)$, $k_3 = k_3 (t)$ so that $k_1 - k_2 = - \lfloor a_2 \sqrt{t} \rfloor$, $k_1 - k_3 =- \lfloor a_3 \sqrt{t} \rfloor$ (and $k_1$ is fixed), and
		let us perform the asymptotic analysis of the expression given by Proposition \ref{prop:2leaders-1} via the steepest descent method, as $t \to \infty$. 
		
		The integrand of \eqref{eq:prop2leaders1} depends on $t$ only through the function $\exp ( t \Phi (v_2,v_3)) v_2^{-a_2 \sqrt{t}} v_3^{-a_3 \sqrt{t}}$, where
		$$
		\Phi (v_2,v_3) := \frac{1}{v_2 v_3} + v_2 + v_3 - 3.
		$$
		A direct computation shows that for $|v_2|=1$ and $|v_3|=1$ we have $\Re \left( \Phi (v_2,v_3) \right) \le 0$, and the equality is reached for $v_2=v_3=1$ only. Also note that $(v_2,v_3)=(1,1)$ is a critical point of $\Phi (v_2,v_3)$. 
		
		Now, let us deform the contours in the expression from \eqref{eq:prop2leaders1} to circles $|v_2|=|v_3|=1 +1/ t^{10}$ (we do not choose exactly $1$ in order to avoid the pole at $v_2 v_3 =1$). Standard for the steepest descent method estimates show that for these contours the leading contribution to the integral comes from an order $t^{-1/2}$ neighborhood of the critical point $(v_2,v_3)=(1,1)$. 
		We make a change of variables $v_3 = 1+ \frac{u_3}{\sqrt{t}}$, $v_2 = 1+ \frac{u_2}{\sqrt{t}}$ in order to compute the integral in this neighborhood. We have
		$$
		\Phi (v_2,v_3) = 2 + \frac{u_3+u_2}{\sqrt{t}} + \frac{1}{1+\frac{u_3+u_2}{\sqrt{t}} + \frac{u_3 u_2}{t} + o(t^{-1}) } -3 = \frac{u_3^2+u_3 u_2 +u_2^2}{t} + o( t^{-1}).
		$$
		Plugging these expressions into the integral, we obtain that in the limit $t \to \infty$ the probability from \eqref{eq:asymp2joint1} is equal to
		\begin{multline*}
			\frac{1}{(2 \pi \ii)^2} \iint \frac{(u_3-u_2)(u_2 + 2 u_3)(2 u_2 + u_3)}{ (u_2 + u_3) u_2 u_3^3 } \\ \times \exp \left( u_2^2 + u_3 u_2 +u_3^2 - a_2 u_2 - a_3 u_3) \right) d u_2 d u_3 + O \left( \frac{1}{\sqrt{t}} \right),
		\end{multline*}
		where the integration contours are $\eps+ \ii \R$ lines, for $0 < \eps \ll 1$. Note that $t^{-1/2}$ appeared thrice in the numerator, five times in the denominator, and two times from Jacobian of the change of variables, which led to the cancellation of this factor in the leading term of the asymptotics.
		After the change of variables $u_2 = \ii x_2$, $u_3 = \ii x_3$, we obtain the following expression:
		\begin{multline*}
			Y(a_2,a_3) := \frac{1}{(2 \pi \ii)^2} \iint \frac{(x_3-x_2)(x_2 + 2 x_3)(2 x_2 + x_3)}{ (x_2 + x_3) x_2 x_3^3 } \\ \times \exp \left( - x_2^2 - x_3 x_2 - x_3^2 - \ii a_2 x_2 - \ii a_3 x_3 \right) d x_2 d x_3,
		\end{multline*}
		where the integration contours are real axes slightly shifted upwards.
		
		The limiting density for the vector 
		$$
		\left( \frac{L_1(t) - k_1}{\sqrt{t}}, \frac{L_2(t) - k_1}{\sqrt{t}} \right)
		$$
		in the $L_2 (t) < L_1(t)$ region is given by
		$$
		\partial_{a_2} \partial_{a_3} Y(a_2,a_3).
		$$
		We claim that this expression is equal to $d_1(a_2,a_3)$ given by \eqref{eq:density2a}. One way to compute the integral is the following: by further differentiation of the integrand with respect to parameters $a_2$ and $a_3$, one can get rid of the denominator and compute the integral in $x_2, x_3$ explicitly. After this, one can take necessary antiderivatives in $a_2, a_3$. We leave details of this computation to math software.
		
		\smallskip
		
		\textbf{Second Step.}
		
		\smallskip
		
		In order to address the $L_2 (t) > L_1(t)$ region, we need to perform asymptotic analysis of the observable from Proposition \ref{prop:2leaders-2}. Performing exactly the same steps as in the first case, we obtain that the leading term of the asymptotics is given by the integral
		\begin{multline*}
			\hat Y (a_2,a_3) := \frac{1}{(2 \pi \ii)^2} \oint \oint \frac{(x_3-x_2)^2 (x_2 + 2 x_3)(2 x_2 + x_3)}{ (x_2 + x_3) x_3^3 } \\ \times \exp \left( - x_2^2 - x_3 x_2 - x_3^2 - \ii a_2 x_2 - \ii a_3 x_3 \right) d x_2 d x_3,
		\end{multline*}
		where again the integration contours are real axes slightly shifted upwards. 
		
		Evaluating this integral in a similar fashion, we arrive at the statement of the proposition. 
		
	\end{proof}
	
	\subsection{Number of leader changes}
	
	\begin{thm}
		\label{th:leader-changes}
		Fix $k \in \Z$. Let $\mathcal{S}(t)$ be the number of changes of the $k$-leader, i.e., the number of moments when the type of the $k$-leader changes. Then
		$$
		\lim_{t \to \infty} \frac{ \mathbf{E} \left( \mathcal{S}(t) \right)}{\ln(t)} = \frac{3 \sqrt{3}}{4 \pi}.
		$$
	\end{thm}
	
	\begin{proof}
		
		We analyze the asymptotics of the expression given by Proposition \ref{prop:number-changes-formula}. 
		
		Note that the function in the exponent is the same as in the previous subsection, so we again apply the steepest descent method with the same deformation of contours, get the leading asymptotic contribution from the neighborhood of  the critical point $(1,1)$, and make a change of variables $v_2 = 1+ u_2 / \sqrt{t}$, $v_3 = 1+ u_3 / \sqrt{t}$ in order to compute it. We pick three factors $1 / \sqrt{t}$ from the numerator, four such factors from the denominator, and two more from the Jacobian of the change of variables. This gives the leading asymptotics $1/\sqrt{t}$. However, note that the integrand is anti-symmetric in the leading order in $t$. Since we take the same contour of integration for $u_2$ and $u_3$, this leading asymptotics also vanishes. The first non-zero contributing term comes from the non-symmetric factor in the integrand, which is the factor $v_2 v_3^2 = 1 + \frac{u_2+2 u_3}{\sqrt{t}}+O(t^{-1})$. Thus, the integral \eqref{eq:number-changes-formula} has leading asymptotics of the form
		$$
		t^{-1} \frac{1}{4 \pi^2} \iint \frac{(u_2-u_3)(u_2+2u_3)(2u_2+u_3)}{(u_2+u_3)^2 u_2 u_3} (u_2+2 u_3) \exp \left( u_2^2 + u_2 u_3 +u_3^2 \right) du_2 du_3,
		$$
		where the contours are $\ii \mathbb{R} + \varepsilon$, for $0<\varepsilon \ll 1$. Making the change of variables $u_i = \ii x_i$, we get
		$$
		t^{-1} \frac{1}{4 \pi^2} \int_{\Im (x_2)= \varepsilon} \int_{\Im (x_3)= \varepsilon} \frac{(x_2-x_3)(x_2+2x_3)^2 (2x_2+x_3)}{(x_2+x_3)^2 x_2 x_3} \exp \left( -x_2^2 - x_2 x_3 - x_3^2 \right) du_2 du_3.
		$$
		This integral can be directly computed, and it is equal to $\frac{3 \sqrt{3}}{4 \pi}$. Integrating in $t$, we arrive at the statement of the theorem.  
		
	\end{proof}
	
	\section{Voter model}
	\label{sec:Voter-section}
	
	\subsection{Totally asymmetric voter model and coalescence process}
	\label{sec:voter-1}
	
	In this section we recall a definition of (particular cases of) a voter model and of a coalescence process on $\Z$.
	
	We consider configurations $\nu:\Z \to \Z$. The equality $\nu (x)=y$ is interpreted as a "voter" at position $x$ having "opinion" $y$. The opinions evolve in time according to the following rules. For each pair on neighboring integers $(x,x+1)$ one associates an independent Poisson process of rate 1 on $\R_{\ge 0}$. Whenever a signal arrives at time $t$, one updates the opinion of a voter at $x$ via $\nu_{t+0} (x) = \nu_{t} (x+1)$. In words, the voter to the left just takes the opinion of its immediate right neighbor. The previous opinion, as well as "the value" of the new opinion, do not play a role.

	In one dimension, there is a well-known projection from the voter model to a coalescence process; let us recall it. Configurations of the coalescence process are maps $\xi: \Z \to \{0,1\}$. Given a configuration of the voter model $\nu$, define $\xi (x) = 1$ iff $\nu (x-1) \ne \nu (x)$, for all $x \in \Z$. In words, we consider connected components of the voters with the same opinion, and mark the left border of such a cluster by a particle.
	
	It is clear that if $\nu_t$ is the totally asymmetric voter model introduced above, the projection $\xi_t$ is Markov and has the following description. Again, for each pair of neighboring positions we do the following updates independently and with rate 1:
	$$
	00 \to 00, \quad 10 \to 10, \quad 01 \to 10, \quad 11 \to 10.
	$$
	These are exactly the update rules of the coalescence process with particles jumping independently with rate 1 to the left only, and coalescing into one particle whenever they are in the same position. 
	
	\begin{ex}
		
		\begin{center}
			\setlength{\tabcolsep}{4pt} 
			\begin{tabular}{l *{14}{>{\small}c}}
				\hline
				\textbf{Voter model $\nu$} &... & >0 & >0 & 0 & 0 & -1 & -1 & -2 & -2 & -2 & <-2 & <-2 & ... &   \\
				\textbf{Coalescence process $\xi$} &... & ... & ... & 1 & 0 & 1 & 0 & 1 & 0 & 0 & ... & ... & ... &   \\
				\textbf{positions} &... & -6 & -5 & -4 & -3 & -2 & -1 & 0 & 1 & 2 & 3 & 4 & ... &   \\
				\hline
			\end{tabular}
		\end{center}
		
		\medskip
		
		In the configurations above, $\nu_t (-1) =-1$, $\xi_t (-1) =0$. If the update between $-1$ and $0$ happens, then in the updated configuration one will have $\nu_{t+0} (-1) =-2$, $\xi_{t+0} (-1) =1$.

	\end{ex}
	
	\subsection{Embedding into the multi-type TASEP}
	\label{sec:Embedding-Voter}
	
	As before, consider a multi-type TASEP $\eta_t$ which starts from the initial configuration $\eta_0(x)=-x$. For each configuration $\eta$, define a map $\tilde \xi: \Z \to \{0,1\}$ by the following rules. Let $\tilde \xi (a) =1$ iff $\eta (x) > \eta (a)$, for all $x < a$, and $\tilde \xi (a) = 0$ otherwise. In words, $\tilde \xi (a) =1$ iff the particle at position $a$ has not interacted with any of the particles which were originally placed on its right (these particles are of smaller type in our notation, and due to non-interaction, they are still all to the right of the particle standing at $a$ at time $t$).
	
	For each configuration $\eta$, define a map $\tilde \nu: \Z \to \Z$ via
	$$
	\tilde \nu(b):= \min_{ x \le b} \eta (x), \qquad \mbox{for $b \in \Z$}.
	$$ 
	In words, this is the minimal type of the particles weakly to the left of $b$. See Example \ref{ex:multTasepEmb} below.
	
	\begin{prop}
		\label{prop:voter-coupling}
		Let $\tilde \xi_t$, $\tilde \nu_t$ be defined according to the rules above from the TASEP $\eta_t$. Then $\tilde \nu_t$ is a totally asymmetric voter process with the initial configuration $\nu_0 (x) = -x$, for all $x \in \Z$, and $\tilde \xi_t$ is a totally asymmetric coalescence process with the initial configuration $\tilde \xi_t (x) =1$, for all $x \in \Z$.
	\end{prop}
	
	\begin{proof}
		Let us consider possible local changes after one update in TASEP, and how they affect their projections $\tilde \xi_t$ and $\tilde \nu_t$; in particular, we will see that these projections are Markov. Let $\eta_t (a)=x > y=\eta_t (a+1)$ be two types of particles, which are updated at the pair of neighboring positions $(a,a+1)$ at time $t$, so $\eta_{t+0} (a) = y$, $\eta_{t+0} (a+1) = x$. Consider the following cases.
		
		\medskip
		
		a) Let $\tilde \xi_t (a) = 0$, $\tilde \xi_t (a+1)=0$. Then one has $z:=\min_{b: b \le a+1} \eta_t (b) < y < x$; it follows that $\tilde \nu_t (a) = \tilde \nu_t (a+1) = z$, and both $\tilde \xi_t$ and $\tilde \nu_t$ do not change after such a TASEP update. 
		
		\medskip
		
		b) Let $\tilde \xi_t (a) = 0$, $\tilde \xi_t (a+1)=1$. Then one has $y=\min_{b: b \le a+1} \eta_t (b)$, thus $\tilde \nu_t (a+1)=y$. After the TASEP update, we have $\tilde \nu_{t+0} (a) = 
		\tilde \nu_{t+0} (a+1)=y$, and $\tilde \xi_{t+0} (a) = 1$, $\tilde \xi_{t+0} (a+1)=0$, since $x>y$.
		
		\medskip
		
		c) Let $\tilde \xi_t (a) = 1$, $\tilde \xi_t (a+1)=0$. Since $x=\min_{b: b \le a} \eta_t (b)$ in that case, and $y<x$, we have $y=\min_{b \in \Z; b \le a+1} \eta_t (b)$, which contradicts $\tilde \xi_t (a+1)=0$. Thus, this case is impossible. 
		
		\medskip
		
		d) Let $\tilde \xi_t (a) = 1$, $\tilde \xi_t (a+1)=1$. In such a case, $x=\min_{b:; b \le a} \eta_t (b)$, and $y=\min_{b: b \le a+1} \eta_t (b)$, which implies $\tilde \nu_t (a)=x$ and $\tilde \nu_t (a+1)=y$. After the update, $x = \eta_{t+0}  (a+1) \ne \min_{b: b \le a+1} \eta_t (b)=y$, which means that $\tilde \xi_{t+0} (a+1) = 0$, while $\tilde \xi_{t+0} (a)=1$, and $\tilde \nu_{t+0} (a) = \tilde \nu_{t+0} (a+1) = y$. 
		
		\medskip
		
		As we see, any update of TASEP updates the processes $\tilde \nu_t$ and $\tilde \xi_t$ in exactly the same way as was described in Section \ref{sec:voter-1}. This completes the proof of the proposition.  
	\end{proof}
	
	For the coalescence process, a similar coupling with the multi-species TASEP was presented in \cite{K}. 
	
	See below for an example of the described map between a TASEP configuration and voter/coalescence process configurations. 
	
	\begin{ex}
		\label{ex:multTasepEmb}
		
		\begin{center}
			\setlength{\tabcolsep}{4pt} 
			\begin{tabular}{l *{14}{>{\small}c}}
				\hline
				\textbf{TASEP $\eta$} & ... & >4 & >4 & 0  & 4 & -1 & 3 & -2 & 1 & 2 & <-2 & <-2 & ... & \\
				\textbf{Voter model $\tilde \nu$} &... & >0 & >0 & 0 & 0 & -1 & -1 & -2 & -2 & -2 & <-2 & <-2 & ... &   \\
				\textbf{Coalescence process $\tilde \xi$} &... & ... & ... & 1 & 0 & 1 & 0 & 1 & 0 & 0 & ... & ... & ... &   \\
				\textbf{positions} &... & -6 & -5 & -4 & -3 & -2 & -1 & 0 & 1 & 2 & 3 & 4 & ... &   \\
				\hline
			\end{tabular}
		\end{center}

	\end{ex}
	
	\medskip
	
	\subsection{Duality with leaders}
	
	In this section we use an algebraic property of the multi-species TASEP in order to relate the voter model with the leaders in TASEP. 
	
	First, let us recall the color-position symmetry of TASEP in our notation. 
	
	\begin{prop}
		\label{prop:cp-symmetry}
		For the multi-type TASEP $\eta_t$ with initial condition $\eta_0 (x)=-x$, one has the following equality in distribution
		$$
		\{ - \eta_t (x) \}_{x \in \Z} \overset{d}{=} \{ \eta_t^{-1} (-x) \}_{x \in \Z}.
		$$
	\end{prop}
	\begin{proof}
		Define a permutation $- \eta_t (x) := \eta (-x)$, $x \in \Z$. Then for $- \eta_t$ one has the color-position symmetry in the form
		$$
		\{ - \eta_t (x) \}_{x \in \Z} \overset{d}{=} \{ (-\eta_t)^{-1} (x) \}_{x \in \Z},
		$$
		see \cite[Theorem 1.4]{AAV}, and also \cite{BW-coloured}, \cite{BB}, \cite{Buf}. We rewrite the right-hand side in order to use it below with the inverse of $\eta_t$ rather than $-\eta_t$.
	\end{proof}

	\begin{prop}
		\label{prop:voter-tasep-dual}
		In the notation of Proposition \ref{prop:voter-coupling},
		let $\mathcal{A}_{y,r_2,r_1}$, for integers $y$ and $r_2 \le r_1$, be the event that $\tilde \nu_t (r_2) = \tilde \nu_t (r_1) = y$ and $\tilde \nu_t (r_2-1) \ne y$. Then 
		$$
		\mathrm{Prob} \left( \mathcal{A}_{y,r_2,r_1} \right) = \mathrm{Prob} \left( \mbox{in TASEP $\eta_t$ the particle of type $r_2$ is $r_1$-leader and is located at $y$} \right).
		$$
	\end{prop}
	
	\begin{proof}
		
		We have:
		\begin{multline}
			\label{eq:voter-duality}
			\mathrm{Prob} \left( \mathcal{A}_{y,r_2,r_1} \right) 
			= 
			\mathrm{Prob} \left( \min \left( \{ \eta_t (x) \}_{x: x \le r_1} \right)
			= \eta_t (r_2) = y \right)
			\\ =
			\mathrm{Prob} \left( \min \left( \{ -\eta_t^{-1} (-x) \}_{x: x \le r_1} \right) = - \eta_t^{-1} (-r_2) = y \right)
			\\ =
			\mathrm{Prob} \left( \max \left( \{ \eta_t^{-1} (-x)  \}_{x: x \le r_1} \right) = \eta_t^{-1} (-r_2) = - y \right)
			=
			\mathrm{Prob} \left( \max \left( \{ \eta_t^{-1} (x) \}_{x: x \le r_1} \right) = \eta_t^{-1} (r_2) = y \right),
		\end{multline}
		where the first equality follows from Proposition \ref{prop:voter-coupling}, the second equality follows from Proposition \ref{prop:cp-symmetry}, and the last equality is the particle-hole duality of TASEP (in the multi-type case, in our notation, it implies that $\{ \eta_t^{-1} (x) \}_{x \in \Z} $ and $\{ -\eta_t^{-1} (-x) \}_{x \in \Z} $ have the same distribution, since they follow the same evolution rules). The event from the last expression of \eqref{eq:voter-duality} is exactly the event that the particle of type $r_2$ is $r_1$-leader and stands at $y$, which completes the proof.   
	\end{proof}
	
	Proposition \ref{prop:voter-tasep-dual} in combination with Theorems \ref{th:leader-type} and \ref{th:leader-conditioned-asymp} immediately implies the following asymptotic results.
	
	\begin{cor}
		\label{cor:asymp-voter}
		Let $\nu_t$ be the voter model which starts from the initial configuration $\nu_0 (x)= -x$, $x \in \Z$, as above. Let $Y_0:= \nu_t (0)$ and  
		$$
		E_0 (t) := \min_{i \le 0} \left(i: \nu_t \left( j \right) =Y_0, \ \mbox{for all $j$ such that $i \le j \le 0$ } \right).
		$$
		In words, $E_0 (t)$ is the left-most position that shares the "opinion" with position $0$. 
		Then we have
		\begin{multline*}
			\lim_{t \to \infty} \left( \frac{E_0(t)}{\sqrt{t}}, \frac{Y_0(t)}{\sqrt{t}}  \right)
			\\ =
			\frac{1}{\sqrt{2 \pi t}} \exp \left( \frac{-y^2}{2} \right) \cdot \frac{1}{2} \left( 1 + \mathrm{erf} \left( \frac{-y-a}{\sqrt{2}} \right) + \left( 1 - \mathrm{erf} \left( \frac{-y}{\sqrt{2}} \right) \right) \exp \left( - a y - \frac{a^2}{2} \right) \right) dy da,
		\end{multline*}
		where the convergence is in distribution and the coordinates are ordered as $(a,y)$. 
	\end{cor}
	
	\begin{cor}
		\label{cor:coalescence}
		Let $\xi_t$ be the coalescence process with a fully packed initial condition $\xi_0 (x)=1$, as above. Let 
		$$
		E_0 (t) := \max_{i \le 0} \left(i: \xi_t \left( i \right) =1, \xi_t \left( j \right) =0, \ \mbox{for all $j$ such that $i < j \le 0$ } \right).
		$$
		In words, $E_0(t)$ denotes the right-most nonpositive occupied position of the coalescence process at time $t$. 
		Then
		$$
		\lim_{t \to \infty} \frac{E_0(t)}{\sqrt{t}} = \frac{1}{\sqrt{\pi}} \exp \left( \frac{-a^2}{4} \right) \mathbf{1}_{a \ge 0} \, da,
		$$
		where the convergence is in distribution. 
	\end{cor}
	
	To our knowledge, Corollaries \ref{cor:asymp-voter} and \ref{cor:coalescence} are new. 
	Similar to Corollary \ref{cor:coalescence} observables of coalescence processes were studied previously in \cite{TZ11}, and Corollary \ref{cor:coalescence} was also independently obtained in \cite{TZ}. 
	
	\begin{rmk}
		In Proposition \ref{prop:voter-tasep-dual} we interpreted our main observable $\M_C(\eta)$ for the case when $\eta$ has two parts only. A similar interpretation (with analogous to Proposition \ref{prop:voter-tasep-dual} proof) also holds for a general $\M_C(\eta)$. 
	\end{rmk}
	
	\section{TASEP ranking process}
	\label{sec:Ranking}
	
	In this section we introduce the \textit{TASEP ranking process} and relate it to leaders in the multi-species TASEP. 
	
	\subsection{Definition}
	
	As before, consider a multi-species TASEP $\eta_t$ which starts from the initial configuration $\eta_0(x)=-x$. For each TASEP configuration $\eta$ define a map $\mathfrak{R}: \Z \to \Z_{> 0}$ via
	$$
	\mathfrak{R} (x) := \mbox{the number of $i \le x$ such that $\eta (i) \le \eta (x)$.}
	$$
	In words, for any position $x$ we look at the type of the particle that stands there, and count how many particles of smaller type stand to the left of it; we also always count the particle at $x$ itself. For the time evolution $\eta_t$, we obtain a process $\mathfrak{R}_t$. Clearly, $\mathfrak{R}_0 (x) = 1$, for all $x \in \Z$. Note that any particle of smaller type to the left of a particle at $x$ must have exchanged with it at some previous moment. Therefore, we call the value $\mathfrak{R}_t (x)$ the \textit{rank} of the particle standing at $x$, and we call the whole process $\mathfrak{R}_t (x)$ the \textit{TASEP ranking process}. 
	
	It is a direct check that this is also a Markov projection of TASEP, which is updated according to the following local rules at neighboring sites: If $ \mathfrak{R}_t (x) = a \ge b = \mathfrak{R}_t (x+1)$, then the update $\mathfrak{R}_{t+0} (x) = b$, $\mathfrak{R}_{t+0} (x+1)=a+1$ happens with rate 1. In the opposite case $ \mathfrak{R}_t (x) < \mathfrak{R}_t (x+1)$ the update is not possible.
	
	Let us remark that if we consider only the evolution of the set $\{x: \mathfrak{R}_{t} (x)=1\}$, we recover the coalescence process $\tilde \xi_t$ from Section \ref{sec:Embedding-Voter}. More generally, for any $k \in \Z_{> 0}$ the process $\mathfrak{R}_{t}$ has a Markov projection $\mathfrak{R}_{t}^{(k)}$ obtained via
	$$
	\mathfrak{R}_{t}^{(k)} (x) := \mathfrak{R}_{t} (x) + \infty \cdot \mathbf{1}_{\mathfrak{R}_{t} (x) >k}.
	$$
	In words, we can replace in the configuration of $\mathfrak{R}_{t} (x)$ all species of type $>k$ by a new type $\infty$; this will not affect the update rules, and will not change the evolution of the particles of types $1,2, \dots, k$. 
	
	\smallskip
	
	\begin{ex}
		
		\begin{center}
			\setlength{\tabcolsep}{4pt} 
			\begin{tabular}{l *{14}{>{\small}c}}
				\hline
				\textbf{TASEP $\eta$} & ... & >4 & >4 & 0  & 4 & -1 & 3 & -2 & 1 & 2 & <-2 & <-2 & ... & \\
				\textbf{Tasep ranking $\mathfrak{R}$} &... & ... & ... & 1 & 2 & 1 & 3 & 1 & 4 & 5 & ... & ... & ... &   \\
				\textbf{Coalescence process $\tilde \xi$} &... & ... & ... & 1 & 0 & 1 & 0 & 1 & 0 & 0 & ... & ... & ... &   \\
				\textbf{positions} &... & -6 & -5 & -4 & -3 & -2 & -1 & 0 & 1 & 2 & 3 & 4 & ... &   \\
				\hline
			\end{tabular}
		\end{center}
		
	\end{ex}
	
	\smallskip

	\subsection{Duality with leaders}
	
	Let $\mathcal{C}_1 (t) := \{ x \in \Z : \mathfrak{R}_{t} (x)=1 \}$ be the positions of the rank one particles in the TASEP ranking process, let $\mathcal{C}_2 (t):= \{ x \in \Z : \mathfrak{R}_{t} (x)=2 \}$ be the positions of the rank two particles in the TASEP ranking process, and $\mathcal{C}_{1,2} (t) := \mathcal{C}_1 (t) \cup \mathcal{C}_2 (t)$. 
	
	\begin{prop}
		\label{prop:ranking-duality}
		Let $k>0, s \ge 0$ be integers. We have
		\begin{multline}
			\label{eq:ranking-duality}
			\mathrm{Prob} \left(\mbox{ $(-k-s) \in \mathcal{C}_1 (t)$, $(-s) \in \mathcal{C}_2 (t)$, $\{ 0, -1, -2, \dots, -s+1, -s-1, \dots, -k-s \} \notin \mathcal{C}_{1,2} (t)$} \right) 
			\\ = \mathrm{Prob} \left( \mbox{particle of type $(k+s)$ is a 0-leader, and paticle of type $s$ is a (0,2)-leader} \right),
		\end{multline}
		where in the event in the right-hand side we, as before, consider TASEP $\eta_t$ which starts from the packed initial condition. 
	\end{prop}
	
	\begin{proof}
		
		Let us denote by $\max^2 (\mathcal{S})$ ( $\min^2 (\mathcal{S})$) the second largest (smallest, respectively) element of a set of numbers $\mathcal{S}$. We have
		\begin{multline*}
			\mathrm{Prob} \left(\mbox{ $(-k-s) \in \mathcal{C}_1 (t)$, $(-s) \in \mathcal{C}_2 (t)$, $\{ 0, -1, -2, \dots, -s+1, -s-1, \dots, -k-s \} \notin \mathcal{C}_{1,2} (t)$} \right) 
			\\ =
			\mathrm{Prob} \left(\min \left( \{ \eta_t (i) \}_{i \le 0} \right) = \eta_t (-k-s), \min\nolimits^2 \left( \{ \eta_t (i) \}_{i \le 0} \right) = \eta_t (-s) \right) 
			\\ = \mathrm{Prob} \left(\min\nolimits \left( \{ -\eta_t^{-1} (-i) \}_{i \le 0} \right) = - \eta_t (-k-s), \min\nolimits^2 \left( \{ -\eta_t^{-1} (-i) \}_{i \le 0} \right) = - \eta_t (-s) \right)
			\\ = \mathrm{Prob} \left(\max\nolimits \left( \{ \eta_t^{-1} (-i) \}_{i \le 0} \right) = \eta_t (-k-s), \max\nolimits^2 \left( \{ \eta_t^{-1} (-i) \}_{i \le 0} \right) = \eta_t (-s) \right)
			\\ = \mathrm{Prob} \left( \mbox{particle of type $(k+s)$ is a 0-leader, and paticle of type $s$ is a (0,2)-leader} \right),
		\end{multline*}
		where the second equality follows from Proposition \ref{prop:cp-symmetry}, while the first and the fourth equalities are rephrasings of definitions. 
	\end{proof}
	
	\begin{rmk}
		Proposition \ref{prop:ranking-duality} allows to extract information about the ranking process from our main observables $\M_C(t)$; in particular, the asymptotics of the right-hand side of \eqref{eq:ranking-duality} is given by Theorem \ref{th:joint}. More general observables of the TASEP ranking process can also be expressed through $\{ \M_C(t) \}$ via similar applications of the color-position symmetry. We do not go into further details in the present paper. 
	\end{rmk}

	
\end{document}